\documentclass{amsart}
\usepackage{amsfonts}
\usepackage{amsmath}
\usepackage{amssymb}
\usepackage{amscd}
\usepackage{color}
\usepackage{amsthm}
\usepackage[hmargin=4cm,vmargin=4cm]{geometry}
\usepackage{parskip}

\usepackage{enumitem}

\newtheorem{thm}{Theorem}

\newtheorem{thm*}{Theorem}
\newtheorem{prop}{Proposition}
\newtheorem{lma}[prop]{Lemma}

\newtheorem{cor}[prop]{Corollary}

\theoremstyle{definition}

\newtheorem{df}[prop]{Definition} 

\newtheorem{conj}[prop]{Conjecture}

\theoremstyle{remark}

\newtheorem{rmk}[prop]{Remark} 

\newtheorem{qtn}[prop]{Question}

\newcommand{\R}{{\mathbb{R}}}
\newcommand{\Z}{{\mathbb{Z}}}
\newcommand{\C}{{\mathbb{C}}}

\newcommand{\bs}{\bigskip}

\newcommand{\del}{\partial}

\newcommand{\sm}[1]{C^\infty(#1)}
\newcommand{\vareps}[1]{\varepsilon_{#1}}

\newcommand\vol{\operatorname{vol}}

\newcommand{\G}{\mathcal{G}}

\newcommand{\cL}{\mathcal{L}}

\newcommand{\til}[1]{\widetilde{#1}}
\newcommand{\wh}[1]{\widehat{#1}}

\newcommand{\om}{\omega}
\newcommand{\al}{\alpha}
\newcommand{\la}{\lambda}

\newcommand{\eps}{\epsilon}

\newcommand{\de}{\delta}

\newcommand{\cA}{\mathcal{A}}

\newcommand{\cG}{\mathcal{G}}
\newcommand{\cH}{\mathcal{H}}

\newcommand{\cR}{\mathcal{R}}

\DeclareMathOperator{\Diff}{\mathrm{Diff}}
\DeclareMathOperator{\Ham}{\mathrm{Ham}}
\DeclareMathOperator{\Cont}{\mathrm{Cont}}

\DeclareMathOperator{\Ker}{\mathrm{Ker}}

\DeclareMathOperator{\osc}{\mathrm{osc}}

\def\H2{H^{(2)}}

\begin{document}

\title[Hofer norm of a contactomorphism]{The Hofer norm of a contactomorphism}
\author{Egor Shelukhin}
\email{egorshel@gmail.com}
\address{Institute for Advanced Study, Einstein Drive, Princeton NJ 08540}


\date{}

\begin{abstract}
We show that the $L^{\infty}$-norm of the contact Hamiltonian induces a non-degenerate right-invariant metric on the group of contactomorphisms of any closed contact manifold. This contact Hofer metric is not left-invariant, but rather depends naturally on the choice of a contact form $\al,$ whence its restriction to the subgroup of $\al$-strict contactomorphisms is bi-invariant. The non-degeneracy of this metric follows from an analogue of the energy-capacity inequality. We show furthermore that this metric has infinite diameter in a number of cases by investigating its relations to previously defined metrics on the group of contact diffeomorphisms. We study its relation to Hofer's metric on the group of Hamiltonian diffeomorphisms, in the case of prequantization spaces. We further consider the distance in this metric to the Reeb one-parameter subgroup, which yields an intrinsic formulation of a small-energy case of Sandon's conjecture on the translated points of a contactomorphism. We prove this Chekanov-type statement for contact manifolds admitting a strong exact filling. 
\end{abstract}

\maketitle


\section{Introduction and main results}




\subsection{Introduction}

Since the advent of the conjugation-invariant Hofer norm \cite{HoferMetric,LalondeMcDuffEnergy} on the group of compactly supported Hamiltonian diffeomorphisms of any symplectic manifold, there has been a certain interest in investigating possible analogues of this norm in contact topology. The first fundamental difference between the Hamiltonian and the contact settings is the finding of Burago-Ivanov-Polterovich \cite{BIP} that classical groups of diffeomorphisms that have a conformal freedom, the full group of diffeomorphisms and the contactomorphism group, do not to admit {\em fine} conjugation-invariant norms - that is norms whose image in $\R_{\geq 0}$ has $0$ as an accumulation point (see \cite{FPR} for the case of the contactomorphism group). Hence, all such norms must be {\em discrete}: there must exist a constant $c>0,$ such that the norm of every diffeomorphism other that the identity transformation exceeds $c.$ And indeed a number of conjugation-invariant non-degenerate norms with values in $\Z$ on groups of contactomorphisms of certain contact manifolds were discovered in \cite{MargheritaMetric,ZapolskyMetric,FPR,ColinSandonMetric} and certain other discrete norms are easily contructed from homogeneous quasi-morphisms on contactomorphism groups found in \cite{GiventalQuasimorphism,BormanZapolsky}.

In a different line of research, one notes that constructions similar to the Hofer metric using contact Hamiltonians tend to yield fine pseudo-norms on contactomorphism groups. In particular, a conjugation-invariant pseudo-norm was introduced on the group of {\em strict} contactomorphisms of any closed contact manifold equipped with a global contact form in \cite{BanyagaDonato} and its non-degeneracy was shown for a certain class of contact manifolds. This non-degeneracy result was improved in \cite{SpaethMullerI} to hold for all closed contact manifolds. By the observation of \cite{BIP,FPR}, this norm cannot extend to a conjugation-invariant norm on the full contactomorphism group. While that is indeed so, the goal of this paper is to show that for all closed contact manifolds, the norm of \cite{BanyagaDonato, SpaethMullerI} is bounded from below by another conjugation-invariant norm that extends to a non-degenerate fine norm on the full group of contactomorphisms, albeit without the property of conjugation-invariance.

The topology induced by this contact Hofer norm, and indeed the equivalence class of the metric it defines, turns out to be independent of the choice of a global contact form, and to admit a number of interesting closed subsets. We show, moreover, that the contact Hofer norm is related to a natural intersection problem for contactomorphisms - that of translated points of S. Sandon \cite{MargheritaTranslatedEarly, MargheritaTranslated}. In particular we show how it gives an intrinsic Chekanov-type statement for the existence of translated points (cf. \cite{ChekanovIntersections,AlbersFrauenfelderLeafwise,AlbersMerryTranslated,AlbersMerryFuchs}). Along the way we describe a geometric non-degeneracy condition for this problem.

We further show how the contact Hofer norm is related to certain previously introduced discrete bi-invariant norms on contactomorphism groups \cite{MargheritaMetric,ZapolskyMetric,FPR,GiventalQuasimorphism,BormanZapolsky} and to the Hofer norm on the group of Hamiltonian diffeomorphisms.

\subsection{Main results}

Given a connected co-oriented contact manifold\footnote{This is what we call a contact manifold in this paper.} $(N,\xi),$ that we will assume to be closed unless stated otherwise, with a globally defined contact form $\alpha$ with $\xi=\ker \al$ we define a contactomorphism to be a diffeomorphism $\psi \in \Diff(N)$ satisfying $\psi_* \xi = \xi$ and preserving the co-orientation of $\xi.$ This is equivalent to the existence of a positive function $\la_\psi = e^{g_\psi},\; g_\psi \in \sm{N,\R}$ such that $\psi^* \al = \la_\psi\cdot \al.$  We consider the group \[\G=\Cont_0(N,\xi)\] of contactomorphisms isotopic to the identity contactomorphism $1.$ 


Recall that $\G$ is naturally isomorphic to the group of $\R_{>0}$-equivariant Hamiltonian diffeomorphisms of the (positive) symplectization $SN$ of $N,$ which is abstractly defined as the subspace \[SN = \{(p,q)|\; p|_{\xi} = 0,\; p > 0_q\} \subset T^*N,\] of the non-zero covectors vanishing on $\xi,$ that are positive with respect to the co-orientation of $\xi,$ endowed with the restriction of the canonical symplectic form $\om_{can} = d\la_{can}$ on $T^*N.$ A choice $\al$ of a global contact form for $(N,\xi)$ the symplectization determines an $\R_{>0}$-equivariantly symplectomorphsm from $SN$ to $N \times \R_{>0}$ with the symplectic form $\om = d(r \til{\al}),$ where $\til{\al}$ denotes the lift of $\al$ by the projection to the first co-ordinate and $r$ denotes the coordinate function coinciding with the inclusion $\R_{>0} \hookrightarrow \R.$ In this isomorphism, the graph of $\al,$ which is a subset of $SN$ corresponds to the hypersurface $N_1:=N \times \{1\} \subset N \times \R_{>0}.$ Given a contactomorphism $\psi \in \G,$ its natural lift $\overline{\psi}$ to $SN$ is simply the restriction to $SN$ of its canonical lift to $T^* N.$ In the splitting $SN \cong N \times \R_{>0}$ given by a global contact form $\al,$ this lift takes the form $\overline{\psi}(x,r) = (\psi(x),\frac{r}{\la_\psi(x)}).$ For later use we note that the differential $D(\overline{\psi})(x,1)$ of $\overline{\psi}$ at $(x,1)$ is 

\begin{equation}\label{Equation: differential of the lift} D(\overline{\psi})(x,1) = \begin{pmatrix}
D({\psi})(x) & 0 \\ -\frac{1}{\la^2_\psi(x)} d \la_\psi(x) & \la_\psi(x)^{-1} 
\end{pmatrix}. \end{equation}

Note that with respect to the natural $\R_{>0}$ action on $SN,$ for an interval $I=(a,b) \subset \R_{>0}$ we have the equality of subsets $I \cdot N_1 = \{r \cdot x\;|\; r \in I,\; x \in N_1\} = N \times I$ with respect to the splitting $SN \cong N \times \R_{>0}.$ For a Hamiltonian $F \in \sm{[0,1] \times SN}$ denote its Hamiltonian flow (generated by the vector field $X_t$ given by $\iota_{X_t} \om = -dF_t$) by $\{\phi^t_F\}_{t \in [0,1]}.$

For an isotopy $\{\psi_t\}_{t \in [0,1]},$  $\psi_0 =1,$ in $\G,$ its time-dependent contact Hamiltonian ${H \in \sm{[0,1] \times N}}$ is by definition the restriction of the corresponding $\R_{>0}$-homogeneous Hamiltonian of degree $1$ on $SN$ to $N_1.$ In other words, for time $t \in [0,1]$ the contact Hamiltonian $H_t(-)=H(t,-)$ at time $t$ satisfies \[H_t=\al(Y_t)\] where $Y_t = \frac{d}{dt'}|_{t'=t} \psi_{t'} \circ \psi_t^{-1}$ is the time-dependent vector field generating $\{\psi_t\}.$ It is easy to see that the contact Hamiltonian determines the flow uniquely (the vector field $Y_t$ is the unique solution of the system $\al(Y_t) = H_t,\; \iota_{Y_t} d\al = - dH_t + dH_t(R_\al) \al$). Denote by $\{\psi_H^t\}$ the flow determined by the contact Hamiltonian $H = H(t,x).$

The one-parameter subgroup $\cR=\cR_\al$ of $\G$ defined by the flow of the autonomous degree-$1$-homogeneous Hamiltonian $H \equiv r$ on $SN \cong N \times \R_{>0}$ is called the {\em Reeb flow} of $\al,$ and the vector field $R=R_\al$ on $N$ generating it is uniquely defined by the conditions $\al(R)=1, \iota_R d\al = 0.$ Hence the natural homomorphism $\R \to \cR$ is given by $t \mapsto \phi_R^t.$

Following S. Sandon \cite{MargheritaTranslatedEarly,MargheritaTranslated}, we call a point $x \in N$ a {\em translated point} of a contactomorphism $\psi$ if $\psi(x) = \phi^\eta_R (x)$ for some $\eta \in \R,$ and $(\psi^* \al)_x = \al_x$ (that is $\la_\psi(x) = e^{g_\psi(x)} = 1$), or alternatively (see \cite{MargheritaTranslatedEarly,MargheritaTranslated}) $x$ is a leaf-wise intersection point of the lift $\overline{\psi}$ of $\psi$ to $SN,$ relative to the hypersurface $N_1 \subset SN.$ In this instance we call $(x,\eta)$ an {\em algebraic translated point} of $\psi$ and note that for a given translated point $x,$ there may be many algebraic translated points covering it (i.e. with first term $x$) \cite{AlbersFrauenfelderLeafwise,AlbersMerryTranslated}. We remark that for any $\psi \in \cR,$ the set of translated points consists of the whole manifold $N.$ 

\medskip
\begin{df}\label{Definition: non-degenerate contactomorphism}
We call an algebraic translated point $(x,\eta)$ of $\psi$ {\em non-degenerate} if the differential \[D(\phi_R^{-\eta}\psi)(x):T_x N \to T_x N\] of $\phi_R^{-\eta}\psi$ at $x$ does not have an eigenvector with eigenvalue $1$ that lies in the kernel of $d\la(x).$ In other words \begin{equation}\label{Equation: non-degeneracy} \ker(D(\phi_R^{-\eta}\psi)(x) - 1) \cap \ker(d\la(x)) = 0.\end{equation}
We say that a contactomorphism $\psi \in \G$ is {\em non-degenerate} if all its algebraic translated points are non-degenerate.
\end{df}

We shall see that this condition is the restriction of the non-degeneracy condition of Albers-Frauenfelder to the class of degree $1$ homogeneous Hamiltonians on $SN.$

\medskip
\begin{lma}\label{Lemma: non-degenerate contactomorphism}
The contactomorphism $\psi$ is non-degenerate if and only if the corresponding Rabinowitz-Floer functional in the symplectization $SN$ of $N$ is Morse.
\end{lma}
\medskip

Hence, whenever $\psi \in \G$ is non-degenerate, the algebraic translated points $(x,\eta) \in N \times \R$ are isolated, and hence are countably many. By an argument of Albers-Frauenfelder \cite{AlbersFrauenfelderLeafwise} that was brought to the contact setting by Albers-Merry \cite{AlbersMerryTranslated}, non-degenerate contactomorphisms are generic.

\medskip
\begin{lma}
The set of contact Hamiltonians generating non-degenerate contactomorphisms is of the second Baire category inside $\sm{[0,1] \times N, \R}.$
\end{lma}

\medskip

Given a global contact form $\al$ for $\xi,$ we denote by $\rho(\al)$ the minimal period of a periodic Reeb orbit of $\al$.

Recall that for a function $H \in \sm{N,\R}$ its $L^\infty$-norm is simply \[|H|_{L^\infty(N)} = \max_N |H|.\]

\medskip
\begin{df}\label{Definition: Hofer energy} Given a global contact form $\al,$ we define for a contactomorphism $\psi \in \G$ and for an element $\til{\psi} \in \til{\G}$ in the universal cover of $\G,$ their Hofer energy  $|\psi|_\al$ and $|\til{\psi}|_\al$ as 
\[ \inf \int_{0}^{1}|H_t|_{L^\infty(N)} dt,\] where the infimum runs over all isotopies $\{\psi_t\}$ with $\psi_1 = \psi$ in the first case and all isotopies in class $\til{\psi}$ in the second case. Clearly, if we denote by $\pi:\til{\G} \to \G$ the canonical projection homomorphism, we have \[|\psi|_\al= \inf_{\pi(\til{\psi}) = \psi} |\til{\psi}|_\al.\]
\end{df}

The first main result of this paper is the non-degeneracy of this Hofer-type functional. We observe certain other properties of this functional, showing that it is a norm on the group $\G.$ Note that this norm is not conjugation-invariant, but instead satisfies an equivariance property which is simply a statement of its naturality with respect to coordinate change. We remark that in the course of preparation of this paper it came to the attention of the author that this functional was previously defined by Rybicki \cite{RybickiI,RybickiII}, and Properties \ref{item: metric Property 2}, \ref{item: metric Property 3}, \ref{item: metric Property 4} other than the non-degeneracy were already observed by him.

\bs
\begin{thm}\label{Theorem: cocycle norm}
The Hofer energy functional $|\psi|_\al$ satisfies the following properties. Denote by $\phi,\psi \in \G$ two arbitrary elements.

\begin{enumerate}[label=(\roman{*}), ref=(\roman{*})]
\item \label{item: metric Property 1}{(non-degeneracy)} If $\psi \neq 1,$ then $|\psi|_\al > 0,$ and $|1|_\al = 0.$
\item \label{item: metric Property 2}{(triangle inequality)} $|\phi\psi|_\al \leq |\phi|_\al + |\psi|_{\al}.$
\item \label{item: metric Property 3}{(symmetry)} $|\psi^{-1}|_\al = |\psi|_\al.$
\item \label{item: metric Property 4}{(naturality)} $|\psi\phi\psi^{-1}|_\al = |\phi|_{\psi^* \al}$
\end{enumerate}

\end{thm}

\bs

\begin{rmk}
The analogous functional on $\til{\G}$ satisfies properties \eqref{item: metric Property 2}, \eqref{item: metric Property 3}, \eqref{item: metric Property 4}. Property \eqref{item: metric Property 1} may not hold, because it is not clear how to bound $|-|_\al$ from below on non-trivial elements of $\pi_1(\G) \subset \til{\G}.$
\end{rmk}

\begin{rmk}\label{Remark: relative distance to Quant}
We remark that these properties imply that $|\cdot|_\al$ descends to a conjugation-invariant norm on the subgroup $\cH = \cH_\al = \Cont_0(N,\al)$ of strict contactomorphisms of the contact form $\al.$ Moreover for $\phi,\psi \in \G$ the distance function (metric) \[d_\al(\phi,\psi) = |\psi \phi^{-1}|_\al\] is invariant with respect to the action of $\G$ on the right, and the action of $\cH$ on the left.
\end{rmk}

\bs
\begin{rmk}\label{Remark: L^p degenerate}
Following Eliashberg-Polterovich \cite{EliashbergPolterovichDeg}, it is easily seen that for $1 \leq p < \infty,$ the pseudo-norm $|-|_{p,\al}$ on $\G$ given by replacing $|H_t|_{L^\infty(N)}$ with $|H_t|_{L^p(N,\al (d\al)^n)}$ in Definition \ref{Definition: Hofer energy} is degenerate. Indeed, call a pseudo-norm $\nu$ on $\G$ {\it quasi-conjugation-invariant} if for each $\phi \in \G$ there exists $C(\phi) > 0,$ such that for all $\psi \in \G,$ \[\nu(\phi \psi \phi^{-1}) \leq C(\phi) \cdot \nu(\psi).\] Then for each such norm the non-degeneracy criterion of Eliashberg-Polterovich holds still: the pseudo-norm $\nu$ is non-degenerate if and only if the $\nu$-displacement-energy \[E_\nu(U) = \inf \{\nu(\phi)\;|\; \phi(U) \cap \overline{U} = \emptyset\}\] of any open subset $U \subset N$ is positive. The proof of this fact is identical to the one in \cite{EliashbergPolterovichDeg}, with the sole difference that the lower estimate on the displacement energy via commutators takes the form \[E_\nu(U) \geq \frac{1}{(C(\phi)+1)(C(\psi) + 1)} \cdot \nu([\phi,\psi])\] for all $\phi,\psi \in \G$ supported in $U.$ One continues to show that $|-|_{p,\al}$ is indeed quasi-conjugation-invariant, and that, by introducing an appropriate time-dependent cutoff, the $|-|_{p,\al}$-displacement energy of each small Darboux ball vanishes. 

Moreover, assuming the simplicity of $\G$ (see \cite{TsuboiSimple,RybickiSimple}), one concludes that $|-|_{p,\al}$ vanishes identically. Indeed $\mathcal{K}_\nu = \{\phi \in \G\;|\; \nu(\phi) = 0\}$ is a normal subgroup of $\G$ for every quasi-conjugation-invariant norm $\nu,$ and in the case $\nu = |-|_{p,\al},$ it is non-empty by the argument above.
\end{rmk}

\bs
\begin{rmk}
The conjugation-invariant norm $(|\cdot|_\al)|_{\cH}$ gives a lower bound for the conjugation-invariant norm $|\cdot |_{str,\al}$ on $\cH$ induced by the restriction of the Finsler metric defining $|\cdot|_\al$ to $\cH,$ that is - we take only paths in $\cH$ in Definition \ref{Definition: Hofer energy}.  Hence $|\cdot |_{str,\al}$ is also a non-degenerate conjugation-invariant norm on $\cH_\al.$ The inequality \[|H_t|_{L^\infty(N)} \leq \osc_N (H_t) + \frac{1}{\vol(N, \al (d\al)^n)} \left |\int_N H_t \al (d\al)^n \right |  \leq\] \[\leq \osc_N (H_t) + \frac{1}{\vol(N, \al (d\al)^n)}\int_N |H_t| \al (d\al)^n \leq 3 |H_t|_{L^\infty(N)}\] where $\osc_N (H_t) = \max_N (H_t) - \min_N (H_t),$ shows immediately an inequality for the corresponding norms \begin{equation}
\label{Equation: BDMS norms}|\phi|_{str,\al} \leq |\phi|_{BD} \leq |\phi|_{MS} \leq 3 |\phi|_{str,\al},\end{equation} where $|\cdot|_{BD}$ is the conjugation-invariant norm of Banyaga-Donato \cite{BanyagaDonato} on $\cH,$ that was reinterpreted and shown to be non-degenerate for all closed contact manifolds by M\"{u}ller-Spaeth in \cite{SpaethMullerI}. The inequality \eqref{Equation: BDMS norms} appeared in \cite{SpaethMullerI}, where it was also shown that $|\phi|_{BD} = |\phi|_{MS}.$
\end{rmk}

\medskip

\begin{rmk}\label{Remark: Calabi-Weinstein}
Note that $|\cdot|_{str,\al}$ on $\til{\cH}_\al$ is always unbounded. Indeed, the Calabi-Weinstein invariant $cw: \til{\cH}_\al \to \R$ (cf. \cite{WeinsteinCalabi}) defined as \[cw([\{\phi^t_H\}]) = \frac{1}{\vol(N, \al (d\al)^n)} \, \int_0^1 dt \int_N H(t,x) \al (d\al)^n\] for any path $\{\phi^t_H\}$ of strict contactormorphisms in a fixed class in $\til{\cH}_\al,$ satisfies \[cw([\{\phi^{\kappa \cdot t}_R\}_{t \in [0,1]}]) = \kappa,\] for all $\kappa \in \R,$ while $|cw(\til{\phi})| \leq |\til{\phi}|_{str,\al}$ for all $\til{\phi} \in \til{\cH}_\al.$ In fact it is easy to see that \[|[\{\phi^{\kappa \cdot t}_R\}_{t \in [0,1]}]|_{str,\al} = \kappa.\]
\end{rmk}

As a metric, $d_\al$ defines a topology on $\G.$ The following lemma states that the equivalence class of the metric $d_\al,$ and consequently this topology, does not depend on the choice $\al$ of a global contact form compatible with the co-orientation of $\xi.$ Note that any two such global contact forms $\al,\al'$ differ by a positive smooth function $f \in \sm{M,\R_{>0}}:$ \[\al' = f \cdot \al.\]

\begin{lma}\label{Lemma: equivalent metrics}
\[\min_N(f) \cdot d_{\al} \leq d_{\al'} \leq \max_N(f) \cdot d_\al, \] and therefore the metrics $d_\al, d_{\al'}$ are equivalent.
\end{lma} 

%


\medskip

The main point of Theorem \ref{Theorem: cocycle norm} is Property \ref{item: metric Property 1}. It follows from the following statement. For a compact subset $A$ of $SN,$ let be the {\em Gromov radius} of $A$ be \[c(A)= \sup\{u |\; \text{there exists a symplectic embedding}\; B(u) \hookrightarrow A \},\] where $B(u) \subset (\R^{2n},\om_{std})$ the standard ball of capacity $u$ (that is radius $\sqrt{\frac{u}{\pi}}$),  and define the {\em height of $A$} to be \[h(A) = h_\al(A) := \sup_A r_\al\] where $r_\al: SN \cong N \times \R_{>0} \to \R_{>0}$ is simply the projection to the second coordinate. Put \[\widehat{c}(A)=\widehat{c}_\al(A):= \frac{c(A)}{h_\al(A)}.\] Note that for $\la \in \R_{>0}$ we have $\widehat{c}(\la \cdot A) = \widehat{c}(A).$

\medskip

\begin{prop}\label{Proposition: energy-capacity}
If the lift $\overline{\psi}$ of $\psi$ to $SN$ displaces a compact subset $A \subset SN,$ then \[|\psi|_\al \geq \frac{1}{4} \widehat{c}_\al(A).\]
\end{prop}

\medskip

\begin{rmk}\label{Remark: extension to non-compact}
Theorem \ref{Theorem: cocycle norm} and Proposition \ref{Proposition: energy-capacity} remain true in the non-compact setting, where $\G$ denotes the identity component $\Cont_{c,0}(N,\xi)$ of the group of compactly supported contactomorphisms of $N.$ 
\end{rmk}

\medskip

\begin{rmk}
Proposition \ref{Proposition: energy-capacity} strengthens \cite[Theorem 1.1]{SpaethMullerI} of M\"{u}ller-Spaeth, combining the use of the energy-capacity inequality of Lalonde-McDuff \cite{LalondeMcDuffEnergy} in the symplectization with techniques like Proposition \ref{Proposition: Usher's trick} and Lemma \ref{Lemma: cutoff}. Consequently a few results in \cite{SpaethMullerI} can be strengthened. For example \cite[Proposition 7.1]{SpaethMullerI} can be strengthened to the following Proposition \ref{Proposition: $C^0$-uniqueness}. It is interesting to see which other applications to questions in $C^0$-contact topology the contact Hofer metric can have. This shall be investigated elsewhere.
\end{rmk}

\medskip

\begin{prop}\label{Proposition: $C^0$-uniqueness}
Assume a sequence $\{\psi_i \in \G\}_{i \in \Z_{>0}}$ of contactomorphisms, a contactomorphism $\psi \in \G$ and a map $\phi: N \to N$ satisfy 
\begin{enumerate}[label=(\roman{*}), ref=(\roman{*})]
\item \label{item: C0 Hofer} $d_\al(\psi_i,\psi) \to 0,$
\item \label{item: C0 uniform} $\psi_i \to \phi$ uniformly, 
\end{enumerate}
as $i \to \infty.$ Then $\phi = \psi.$
\end{prop}

\medskip

For a subset $D \subset N$ with non-empty interior, define its {\em contact $\al$-capacity} as \[\til{c}_\al(D) = \wh{c}_\al(\pi^{-1}(D)) = \sup_{A \subset \pi^{-1}(D)} \wh{c}_\al(A) > 0,\] the supremum running over all compact subsets $A \subset \pi^{-1}(D).$ Moreover, we say that $\psi \in \G$ {\em displaces} $D$ if $\psi(D) \cap \overline{D} = \emptyset,$ and define the $\al$-displacement energy of $D$ to be \[E_\al(D) = \inf\{\,|\psi|_\al \;|\; \psi\; \text{displaces}\; D\}.\]
Then Proposition \ref{Proposition: energy-capacity} immediately implies the following $\al$-dependent {\em energy-capacity inequality} on $N.$ 

\medskip

\begin{cor}\label{Corollary: contact energy-capacity} \[E_\al(D) \geq \frac{1}{4}\til{c}_\al(D).\]\end{cor}

\begin{rmk}\label{Remark: contact capacity C0 continuous}
Note that in fact, at least for open $D \subset N,$ $\til{c}_\al(D) = c(\pi^{-1}(D) \cap \{r<1\}),$ and therefore $\psi \mapsto \til{c}_\al(\psi(D))$ is a $C^0$-continuous function on $\G.$
\end{rmk}

\medskip

\begin{rmk}\label{Remark: contact capacity bounded}
It is straightforward to show that in case when $N$ is closed the contact $\al$-capacity $\widehat{c}_\al$ (and hence $\widetilde{c}_\al$) is universally bounded. Indeed comparing the Gromov radius $c$ with volume we have for any compact subset $A \subset SN$ the estimate \[\widehat{c}_\al(A) \leq \frac{\pi}{(v_{2n+2})^{1/{n+1}}} \vol(N,\al (d\al)^n)^{1/{n+1}},\] where $v_{2n+2}$ is the volume of the unit ball $B \subset \R^{2n+2}.$ In the non-compact setting this universal bound disappears, being replaced by \[\widehat{c}_\al(A) \leq \frac{\pi}{(v_{2n+2})^{1/{n+1}}} \vol(\pi(A),\al (d\al)^n)^{1/{n+1}},\] where $\pi:SN \to N$ is the natural projection.
\end{rmk}

\medskip

Remark \ref{Remark: contact capacity bounded} brings us to the following interesting question.

\medskip
\begin{qtn}\label{Question: unbounded}
For which contact manifolds with a global contact form $(N,\xi,\al)$ is the norm $|\cdot|_\al$ unbounded, that is \[\sup_{\psi \in \G} |\psi|_\al = \infty?\]
\end{qtn}

An example of such a manifold is a standard prequantization space $(P_1,\al_1)$ of the two-torus $(T^2,\om_{std})$ \cite{KirillovQuantization}. For more partial results on this question see Proposition \ref{Proposition: upper bound}. Put $\xi_1 = \Ker \al_1$ for the contact structure on $P_1.$

\medskip
\begin{prop}\label{Proposition: example of unboundedness}
Let $H \in \sm{T^2,\R}$ be an autonomous normalized Hamiltonian with the non-contractible closed curve $\{0\} \times S^1$ in $T^2$ as a component of a regular level set. Consider contact flow $\{\psi^t=\psi^t_{\pi^* H}\}_{t>0}$ generated by the lift $\pi^* H$ to $P_1$ of the Hamiltonian $H.$ This flow is a canonical lift to $P_1$ of the Hamiltonian flow of $H.$ Then \[|\psi^t|_{\al_1} \geq \mathrm{const} \cdot t,\] for a constant $\mathrm{const} >0.$
\end{prop}

%

\medskip
\begin{rmk}
The same argument works for prequantizations of a surface $\Sigma$ of higher genus and a simple closed curve $\gamma$ of infinite order in $\pi_1(\Sigma)$ instead of $(T^2,\{0\} \times S^1).$ Moreover, a very similar argument applies to show that $|[\{\phi^{s \cdot t}_R\}_{s \in [0,1]}]|_\al \geq \mathrm{const} \cdot t,$ for $t > 0.$
\end{rmk}

\medskip
Similarly to the non-degeneracy of the contact Hofer metric, we deduce from Proposition \ref{Proposition: energy-capacity} the following basic topological property of $\cH$ in the topology defined by the distance function $d_\al.$

\bs
\begin{prop}\label{Proposition: Quant is closed}
If $\psi \in \G \setminus \cH,$ then there exists $\varepsilon = \vareps{\psi,\al}>0,$ such that $d_\al(\psi,\phi) \geq \varepsilon$ for all $\phi \in \cH.$ In other words, $\cH_\al$ is closed in the topology defined by $d_\al.$
\end{prop}

This proposition establishes the lower bound \[d_\al(\psi,\cH_\al) = \inf_{\phi \in \cH} d_\al(\psi,\phi) \geq \varepsilon,\] raising the following question.

\medskip
\begin{qtn}\label{Question: contactomorphisms far from Quant}
For which contact manifolds $(N,\xi)$ with a global contact form $\al$ there are contactomorphisms $\psi \in \G$ that are arbitrarily $d_\al$-far from $\cH_\al,$ or in other words, when \[\sup_{\psi \in \G} d_\al(\psi,\cH_\al) = \infty?\]
\end{qtn}

This question, for a different metric \cite{ZapolskyMetric} on $\G = \Cont_{0,c}(T^* B \times S^1)$ where $B$ is a closed connected manifold of positive dimension, was asked by F. Zapolsky \cite{ZapolskyPrivate}. 

\medskip
Next we turn to a different functional on the group $\G$ and on its universal cover $\til{\G},$ coming from a different way to express the Hofer norm in the usual Hamiltonian setting.  

\bs
\begin{df}\label{Definition: Osculation norm}
Given a global contact form $\al,$ we define for a contactomorphism $\psi \in \G$ and for an element $\til{\psi} \in \til{\G}$ in the universal cover of $\G,$ their oscillation Hofer energy  $|\psi|^{\osc}_\al$ and $|\til{\psi}|^{\osc}_\al$ as \[ \inf \int_{0}^{1}\osc_N(H_t)dt,\] where \[\osc_N(H_t) = \max_N(H_t) -\min_N(H_t),\] and the infimum runs over all isotopies $\{\psi_t\}$ with $\psi_1 = \psi$ in the first case and all isotopies in class $\til{\psi}$ in the second case. 
\end{df}

It is curious that while this oscillation energy functional does not turn out to be equivalent to the $L^\infty$ energy functional like in the Hamiltonian case, it does, in contrast to the Hamiltonian case, turn out to have a closed form expression in terms of the $L^\infty$ energy.
\bs
\begin{prop}\label{Proposition: Osc via L infty}
For each $\psi \in \G,$ \[\frac{1}{2}|\psi|^{\osc}_\al = d_\al(\psi,\cR) = \inf_{t \in \R} d_\al(\psi,\phi^t_R).\] 
\end{prop}


This proposition has the following corollary.

\medskip
\begin{cor}\label{Corollary: maximal non-degeneracy of Osc}
For $\psi \in \G,$ $|\psi|^{\osc}_\al = 0$ if and only if $\psi$ belongs to the closure $\overline{\cR}$ of $\cR \subset \cG$ in the topology defined by $d_\al.$
\end{cor}

This corollary brings forth the following question. 
\medskip

\begin{qtn}\label{Question: is Reeb closed}
For which contact forms $\al,$ $\overline{\cR}=\cR,$ that is $\cR_\al$ is a closed subgroup of $\G$ in the topology defined by $d_\al$?
\end{qtn}

\medskip
\begin{rmk}\label{Remark: Reeb closed} It is immediate that those contact manifolds with a global contact form for which the Reeb flow is periodic, for example prequantization spaces, satisfy this property. Moreover, by Proposition \ref{Proposition: Quant is closed}, we see that $\overline{\cR} \subset \cH,$ edging closer to an answer to this question. Indeed by an argument of M\"{u}ller-Spaeth \cite{SpaethMullerIII} and Casals-Spacil \cite{CasalsSpacil}, whenever the Reeb flow $\{\phi^t_R\}$ of $\al$ has a dense orbit, we have $\cH = \cR,$ implying $\overline{\cR}=\cR.$ In the case of irrational ellipsoid with the standard contact form, it is easy to see that $T = \overline{\cR}$ is a torus and $\cR \subset T$ is an irrational one-parametric subgroup. It would be interesting to study the closure of $\cR$ in more general situations. \end{rmk} 

In \cite{MargheritaTranslatedEarly,MargheritaTranslated} Sandon introduced the notion of a translated point of a contactomorphism $\psi \in \G$ and observed that the number of translated points of a $C^1$-small (i.e. $C^1$-close to the identity transformation) contactomorphism $\psi$ is at least the minimal number of critical points of a function on $N$ and for generic such $\psi,$ it is at least the minimal number of critical points of a Morse function on $N.$ Consequently she conjectured the same conclusion to hold for all $\psi \in \G.$ Here we study the following weak homological version of Sandon's conjecture.

\medskip

\begin{conj}\label{Conjecture: Sandon}
Every contactomorphism $\psi \in \G$ has a translated point, and generic $\psi \in \G$ have at least $\dim H_*(N,\Z/(2))$ translated points.
\end{conj}

The Sandon conjecture or its homological version has been shown to hold in a number of cases in \cite{MargheritaTranslatedEarly,MargheritaTranslated,AlbersMerryTranslated,AlbersMerryFuchs}. We remark that the genericity assumption, similarly to the case of the Arnol'd conjecture on Hamiltonian diffeomorphisms, should correspond to a non-degeneracy condition on the diffeomorphism $\psi$, for which our notion from Definition \ref{Definition: non-degenerate contactomorphism} is a candidate.

Proposition \ref{Proposition: Osc via L infty} suggests that if the oscillation energy of $\psi$ is small enough, then it should have translated points. To this end we have the following definition and theorem.

\medskip
\begin{df}\label{Definition: small oscillation energy}
We say that $\psi \in \G$ has {\em small oscillation energy} if $\frac{1}{2}|\psi|^{\osc}_\al < \rho(\al).$
\end{df}

\bs
\begin{thm}\label{Theorem: translated points filling.}
Assume that $(N,\xi)$ admits a strong exact symplectic filling $(M,\tau = d\overline{\beta}),$ with $\alpha = \overline{\beta}|_N$ a global contact form for $(N,\xi).$ Then every contactomorphism $\psi$ of small oscillation energy has at least one translated point, and if in addition $\psi$ is non-degenerate, then it must have at least $\dim H_*(N,\Z/(2))$ translated points.
\end{thm}
 
\medskip

Theorem \ref{Theorem: translated points filling.} is an analogue of, and indeed follows from the Chekanov-type theorems \cite[Theorem A, Theorem B]{AlbersFrauenfelderLeafwise} of Albers-Frauenfelder for the case of leafwise intersection points of Hamiltonian flows in an exact symplectic manifold $(W,d\beta)$ relative to bounding hypersurfaces $\Sigma$ of restricted contact type ($\al=\beta|_\Sigma$ is a contact form). A point $x \in \Sigma$ is a leafwise intersection for $\phi \in \Ham_c(W,d\beta)$ relative to $\Sigma$ if $\phi (x) \in \Sigma$ and moreover $\phi (x)$ lies in the same leaf as $x$ of the characteristic foliation $\ker (d\beta|_\Sigma),$ that is - on the same orbit as $x$ of the Reeb flow of $\al.$

\medskip
\begin{thm}[(Albers-Frauenfelder \cite{AlbersFrauenfelderLeafwise})]\label{Theorem: Albers-Frauenfelder}
If $\phi \in \Ham_c(W,d\beta)$ has Hofer norm $|\phi|_{Hofer} < \rho(\al),$ then $\phi$ has a leafwise intersection point relative to $\Sigma,$ and generic such $\phi$ have at least $\dim H_*(\Sigma,\Z/(2))$ leafwise intersection points.
\end{thm}

The paper \cite{AlbersFrauenfelderLeafwise} was used by Albers-Merry in \cite{AlbersMerryTranslated} to prove results on translated points of contactomorphisms of contact manifolds with a strong exact filling. Techniques like Lemma \ref{Lemma: cutoff} allow us to improve their estimates on the Hofer energy and prove a sharper Chekanov-type theorem in this setting. We also note that when $\rho(\al) = +\infty,$ namely there are no closed Reeb orbits, it was shown in \cite{AlbersMerryFuchs} that Conjecture \ref{Conjecture: Sandon} holds without the requirement on the existence of a symplectic filling. This result of Albers-Fuchs-Merry and Theorem \ref{Theorem: translated points filling.} constitute reasonable evidence towards the following conjecture, removing the assumption on the existence of a strong exact symplectic filling. 

\bs
\begin{conj}\label{Conjecture: translated points general weak.}
Every contactomorphism $\psi$ of small oscillation energy has at least one translated point, and if in addition $\psi$ is non-degenerate, then it must have at least $\dim H_*(N,\Z/(2))$ translated points.
\end{conj}

The methods of \cite{AlbersMerryFuchs} seem to give a version of this conjecture (cf. Albers-Hein \cite{AlbersHein}) which is weaker in terms of assumptions. Additionally, following \cite{AlbersHein}, one expects a cuplength estimate instead of the estimate $ \geq 1$ in the degenerate case, strengthening the conclusion.

We proceed by studying the relationship of the contact Hofer norm with a number of conjugation invariant norms on the group $\G$ and $\til{\G}.$ The norms that were constructed on the group $\G$ are the norm $\nu_S$ of Sandon \cite{MargheritaMetric} on $\G = \Cont_{0,c}(\R^{2n} \times S^1)$ and $\nu_Z$ of Zapolsky \cite{ZapolskyMetric} on $\G = \Cont_{0,c}(T^* B \times S^1)$ for a closed connected manifold $B$ of positive dimension (remark, we take $\nu_Z = \rho_{\sup}$ in the notations of \cite{ZapolskyMetric}). The norms that were constructed on the group $\til{\G}$ are $\nu_{CS}$ by Colin-Sandon \cite{ColinSandonMetric} for any contact manifold and $\nu_{FPR}$ by Fraser-Polterovich-Rosen \cite{FPR} for orderable contact manifolds whose Reeb flow generates a circle action (note: we only consider the subgroup $\til{\G}=\til{\G_c}$ of $\til{\G}_e$ in the notation of \cite{FPR}). Moreover, the quasimorphisms $\mu_G$ of Givental \cite{GiventalQuasimorphism} on $\til{\G}$ for $\G=\Cont_0(\R P^{2n-1},\xi_{std})$ and $\mu_{BZ}$ of Borman-Zapolsky \cite{BormanZapolsky} on $\til{\G}$ for certain toric contact manifolds give norms $\nu_{G},\nu_{BZ}$ on $\til{\G}$ defined by $\nu_G = |\mu_G| + D(\mu_G)$ and $\nu_{BZ} = |\mu_{BZ}| + D(\mu_{BZ})$ on $\til{\G}\setminus \til{1},$ where $D(\mu) \geq 0$ denotes the defect of the quasimorphism $\mu,$ and $\nu_{G}(\til{1})=0, \nu_{BZ}(\til{1})=0.$ It turns out that $|\cdot|_\al$ provides a sort of universal upper bound for most of these metrics, and is therefore unbounded whenever these metrics are unbounded. We list these estimates below, the proofs of which are either straightforward computations or are contained in existing work and hence are mostly omitted. However we note the following bound which facilitates these computations.

\medskip

\begin{lma}\label{Lemma: ceiling of the norm}
For $\phi \in \G$ and $\til{\phi} \in \til{\G}$ the numbers  $\lceil|\phi|_\al\rceil$ and $\lceil|\til{\phi}|_\al\rceil$ are bounded from below by \[\min(\max\{|\lceil \int_0^1 \max_N H_t \,dt \rceil|,|\lfloor \int_0^1 \min_N H_t \,dt\rfloor|\}) - \epsilon\] where the minimum is taken over all contact Hamiltonians $H$ whose flow $\{\phi^t_H\}_{t \in [0,1]}$ satisfies $\phi = \phi^1_H$ in the first case, and $[\{\phi^t_H\}] = \til{\phi} \in \til{\G}$ in the second case, and $\epsilon \in \{0,1\}$ is $1$ if and only if $|\phi|_\al$ (repsectively $|\til{\phi}|_\al$) is a non-negative integer, and the infimum in the definition of the contact Hofer norm is not attained. 
\end{lma}

\medskip
Lemma \ref{Lemma: ceiling of the norm} is a consequence of the inequality \[ \lceil \int_0^1 |H_t|_{L^\infty(N)}\,dt \rceil \geq  \max\{|\lceil \int_0^1 \max_N H_t \,dt \rceil|,|\lfloor \int_0^1 \min_N H_t \,dt\rfloor|\}\] which follows by straightforward case-by-case analysis, and the semi-continuity properties of the ceiling function $\lceil \cdot \rceil.$

\bs
\begin{prop}\label{Proposition: upper bound}
For all $\phi \in \G, \til{\phi} \in \til{\G},$ we have the estimates 
\begin{enumerate}[label=(\roman{*}), ref=(\roman{*})]
\item \label{item: upper bound 1} $\nu_S(\phi) \leq 2\lceil|\phi|_{\la_{std}}\rceil,\; $ where $\la_{std}$ is the standard contact form on $\R^{2n}\times S^1,$
\item \label{item: upper bound 2} $\nu_Z(\phi) \leq \lceil |\phi|_{\la_{std}} \rceil,\;$ where $\la_{std}$ is the standard contact form on $T^* B\times S^1,$
\item \label{item: upper bound 3} $\nu_{FPR}(\til{\phi}) \leq \lceil |\til{\phi}|_{\la} \rceil + 1,\;$ where $\la$ is the contact form in the notation of \cite{FPR},
\item \label{item: upper bound 4} $\nu_G(\til{\phi}) \leq \mathrm{const} \cdot (\lceil |\til{\phi}|_{\al_{std}}\rceil +1) + D(\mu_G),\;$ where $\mathrm{const} = \mu_G([\{\phi^t_R\}_{t \in [0,1]}]),$
\item \label{item: upper bound 5} $\nu_{BZ}(\til{\phi}) \leq \mathrm{const}\cdot (\lceil|\til{\phi}|_{\al_{std}}\rceil+1) + D(\mu_{BZ}),\;$ where $\mathrm{const} = \mu_{BZ}([\{\phi^t_R\}_{t \in [0,1]}]).$
\end{enumerate}

Hence $|\cdot|_\al$ is unbounded whenever any one of the norms $\nu_S, \nu_Z, \nu_{FPR}, \nu_{G}, \nu_{BZ}$ is unbounded.
\end{prop}

\medskip
Item \ref{item: upper bound 2} is simply the upper bound on $\rho_{\sup}(\phi)$ from \cite[Proof of Theorem 1.3]{ZapolskyMetric}, and item \ref{item: upper bound 1} follows by tracing the identifications of contact manifolds in \cite{MargheritaMetric} from the same upper bound and a comparison between $\rho_{\osc}$ and $\rho_{\sup}$ in the notations of \cite{ZapolskyMetric}. Item \ref{item: upper bound 3} follows by the identities \[\nu_+(\til{\phi}) = \min_{[\{\phi^t_H\}]=\til{\phi}} \,\lceil \int_0^1 \max H_t \,dt \rceil,\; \nu_-(\til{\phi}) = \max_{[\{\phi^t_H\}]=\til{\phi}} \,\lfloor \int_0^1 \min H_t \,dt \rfloor,\]\[\nu_{FPR} = \max \{|\nu_+|,|\nu_-|\},\] in the notations of \cite{FPR} and Lemma \ref{Lemma: ceiling of the norm} by straightforward case-by-case analysis. Items \ref{item: upper bound 4} and \ref{item: upper bound 5} follow from Item \ref{item: upper bound 3} and \cite[Lemma 1.33]{BormanZapolsky}, the coefficient $\mathrm{const}$ being the one that appears in \cite[Lemma 1.33]{BormanZapolsky}.

Proposition \ref{Proposition: upper bound} has the following consequence which is similar in spirit to \cite[Proposition 3.11]{FPR}. We call an open subset $U \subset N$ of $N$ {\em downscalable} if for every compact subset $A \subset U$ and for every $K > 0$ there exists $\psi=\psi_K \in \Cont(N),$ such that $\la_{\psi} = e^{g_{\psi}}$ satisfies $g_{\psi}|_A \leq -K.$ Note that the property of $U$ being downscalable does not depend on the choice of the contact form $\al$ on $N.$
A simple example of a downscalable $U$ is the complement of a point in $S^1$ with the standard contact structure. A more general example is any contact Darboux ball, that is - the image of a contact embedding of a standard contact ball \begin{equation}\label{Equation: contact ball} B_{st} = \bigg\{\sum_j (x_j^2+y_j^2) + {z^2} < 1\bigg\}\end{equation} in the standard contact $\R^{2n} \times \R$ with coordinates $(x_1,y_1,\ldots,x_n,y_n,z)$.


\medskip

\begin{cor}\label{Corollary: downscalable}
Any element $\til{\phi} \in \til{\G}$ that lies in the image of the natural map $\til{\Cont}_{0,c}(U) \to \til{\G}$ for a downscalable subset $U$ has \[\inf_{\psi \in \G} |\psi \til{\phi} \psi^{-1}|_{\al} = 0.\] Hence $\nu_{FPR}(\til{\phi}) \leq 2,$ and similar statements hold for the other norms $\nu_S,\nu_Z, \nu_G, \nu_{BZ}.$
\end{cor}

\begin{proof}
Let $Y(t,x)$ be a contact vector field generating $\til{\phi}$ such that $Y(t,x) = 0$ outside a compact $A \subset U.$ Choose $K>0,$ and let $\psi$ be such that $\la_\psi \leq e^{-K}$ on $A.$ For a choice of global contact form $\al,$
set $l(Y,\al) = \int_0^1 |H_t|_{L^\infty(N)} \,dt,$ where $H=\al(Y)$ is the contact Hamiltonian of the vector field $Y.$ Note that by Property \ref{item: metric Property 4} of the contact Hofer metric, $|\psi \til{\phi} \psi^{-1}|_{\al} = |\til{\phi}|_{\psi^* \al}.$ Moreover, since $\psi^*\al = \la_\psi \al,$ \[l(Y,\psi^*\al) = \int_0^1 |\la_\psi H_t|_{L^\infty(N)} \,dt, =  \int_0^1 |\la_\psi H_t|_{L^\infty(U)} \,dt \leq e^{-K} l(Y,\al).\] Taking infima now shows that \[\inf_{\psi \in \G} |\til{\phi}|_{\psi^* \al} \leq e^{-K} l(Y,\al)\] for all $K>0.$ This finishes the proof of the first statement. The second statment is now an immediate corollary of Proposition \ref{Proposition: upper bound} and the conjugation-invariance of the norm $\nu_{FPR}.$
\end{proof}


\begin{rmk}
Recall that the the Lie algebra of $\G$ is identified with $\sm{N,\R}$ by a choice of a contact form $\al,$ and under this identification, the adjoint action of $\G$ on its Lie algebra takes the form $\mathrm{Ad}_\psi(H) = (\la_\psi \, H) \circ \psi^{-1}.$ It is curious to note that by the same argument as in the proof of Corollary \ref{Corollary: downscalable}, given any norm $|\cdot|_0$ that is invariant under pull-back by elements of $\G$ on the linear space $\sm{N,\R}$ (for example the $L^\infty$-norm), its $\mathrm{Ad}$-invariantization (cf. \cite{YaronLev}) 
\[|H|_0^{Ad}= \inf_{\substack{H_1 + \ldots + H_l = H,\\ \psi_1,\ldots, \psi_l \in \G}} |\mathrm{Ad}_{\psi_1} (H_1)|_0 + \ldots + |\mathrm{Ad}_{\psi_l} (H_l)|_0,\] yields an ($\mathrm{Ad}$-invariant!) pseudo-norm that is identically zero. This confirms the intuition that conjugation-invariant norms on $\G$ are objects of global topological nature. 
\end{rmk}

Considering conjugacy classes of elements in $\G$ leads one to the following question. 

\medskip
\begin{qtn}\label{Question: sup over conjugacy classes}
For which contact manifolds $(N,\xi)$ with contact form $\al$ the value \[|\phi|_{conj,\al} = \sup_{\psi \in \G} |\psi \phi \psi^{-1}|_{\al}\] is finite for all $\phi \in \G$?
\end{qtn}

Note that whenever this finiteness condition holds, $|\cdot|_{conj,\al}$ gives a conjugation-invariant non-degenerate norm on $\G.$ Clearly there is an analogous definition for $\til{\G}.$ An example of $\phi$ with $|\phi|_{conj,\al} = + \infty$ would have to involve a sequence $\{\psi_j\}$ of contactomorphisms such that $\max \la_{\psi_j} \to \infty$ and $|\psi_j|_\al \to \infty$ as $j \to \infty,$ and hence a partial answer to Question \ref{Question: unbounded}. In particular whenever $|\cdot|_\al$ is bounded, we obtain a bounded conjugation-invariant norm on $\G.$ Moreover, according the theorem of \cite{FPR}, there must exist a constant $c>0,$ such that $|\phi|_{conj,\al} \geq c$ for all $\phi \neq 1.$ The following is a numerical bound on $c>0$.

\medskip
\begin{prop}
For all $\phi \neq 1,$ \[|\phi|_{conj,\al} \geq \frac{1}{4} c_0,\]  \[c_0 = \sup_B \til{c}_\al(B) > 0,\] where $B$ runs over all contact embeddings of the standard contact ball $B_{st}$ into $N.$
\end{prop}  

\begin{proof}
For a subset $A \subset N,$ denote $\overline{c}_\al(A) = \sup_{\psi \in \G} \til{c}_\al (\psi(A)).$ By the energy-capacity inequality (Corollary \ref{Corollary: contact energy-capacity}) we have \[|\phi|_{conj,\al} \geq \frac{1}{4} \overline{c}_\al(B_0)\] where $B_0$ is a small contact ball displaced by $\phi.$ We claim that \[\overline{c}_\al(B_0) \geq c_0.\]

Let $B$ be the image of a contact embedding $\iota: B_{st} \to N.$ Note that there exists a contactomorphism $\psi_1 \in \G$ mapping the center $p_0$ of $B_0$ to the center $p=\iota(0)$ of $B.$ Then the image of $B_0$ under $\psi_1$ contains the image $\iota(B(\eps))$ of a small ellipsoid $B(\eps) = \{\sum_j \frac{(x_j^2+y_j^2)}{\eps^2} + \frac{z^2}{\eps^4} < 1\}$ around $0.$ Hence \begin{equation}\label{Equation:cap}\overline{c}_\al(B_0) \geq \overline{c}_\al(\iota(B(\eps))).\end{equation} Moreover, given a number $0 <\rho < 1,$ considering the autonomous contact flow $(x,y,z) \mapsto (e^\tau \cdot x, e^\tau \cdot y,e^{2\tau} \cdot z)$ on $\R^{2n} \times \R,$ with contact Hamiltonian cut off sufficiently close to $\partial B_{st},$ we construct a contactomorphism $\psi_{2,\eps,\rho,st}$ supported in $B_{st}$ and mapping $B(\eps)$ to $B(\rho).$ Clearly this contactomorphism extends to a contactomorphism $\psi_{2,\eps,\rho}$ of $N.$ This, and the $C^0$-continuity of the contact capacity $\til{c}_\al$ (Remark \ref{Remark: contact capacity C0 continuous}) gives the equality \[\overline{c}_\al(\iota(B(\eps))) = \overline{c}_\al (B),\] which combined with (\ref{Equation:cap}) finishes the proof.
\end{proof}

\medskip

We conclude this section by studying the relation of the contact Hofer norm in the case of prequantization spaces and the usual Hofer norm on the group of Hamiltonian diffeomorphisms of a symplectic manifold. Given a symplectic manifold $(M,\om),$ the group $\Ham(M)$ of Hamiltonian diffeomorphisms of $(M,\om)$ consists of the time-$1$-maps $\phi^1_H$ of time-dependent Hamiltonian vector fields $X_H$ obtained by the Hamiltonian construction $\iota_{X_H} \om = - dH$ from compactly supported Hamiltonian functions $H \in C^\infty_c([0,1] \times M,\R).$ In the case when $M$ is closed it is enough to consider only $H \in C^\infty_0([0,1] \times M,\R)$ normalized by the condition that $\int_M H(t,-) \om^n = 0$ for all $t \in [0,1].$

The Hofer norm on $\Ham(M)$ was introduced by Hofer \cite{HoferMetric} and shown to be always non-degenerate by Lalonde-McDuff \cite{LalondeMcDuffEnergy}. It is defined as \[|\phi|_{Hofer} = \inf \int_0^1 (\max_M H(t,\cdot) - \min_M H(t,\cdot)) \,dt,\] where the infimum runs over all $H \in C^\infty_0([0,1] \times M,\R)$ with $\phi^1_H = \phi.$ Alternatively, one can take the infimum \[|\phi|'_{Hofer}=\inf \int_0^1  |H(t,\cdot)|_{L^\infty(M)} \,dt\] over the same set of $H,$ which results in an equivalent norm: \[\frac{1}{2} |\cdot|_{Hofer} \leq |\cdot|'_{Hofer} \leq |\cdot|_{Hofer}.\]

Let $(P,\al)$ be a prequantization space of the symplectic manifold $(M,\om).$ That is $d\al = \til{\om},$ for the lift $\til{\om} = p^* \om$ of $\om$ to $P$ by the fibration map $p: P \to M.$ Such a prequantization exists if and only if the class $[\om] \in H^2(M,\R)$ is integral, that is lies in $\text{Image}(H^2(M,\Z) \to H^2(M,\R)).$ Then the group $\cH=\cH_\al \subset \G$ of strict contactomorphisms (called quantomorphisms in this setting) and the group $\Ham=\Ham(M)$ enter the following central group extensions:
\begin{gather}\label{Equation: Quant Ham long exact}
1 \to S^1 \to \cH \xrightarrow{pr} \Ham \to 1,\\
0 \to \R \to \til{\cH} \xrightarrow{\til{pr}} \til{\Ham} \to 0.
\end{gather}

\medskip

Moreover the pullback map $p^*: C^\infty_0(M,\R) \to \sm{P,\R},$ $H \mapsto H\circ p$ on functions induces a natural homomorphism \[i:\til{\Ham} \to \til{\cH}\] which is a section for $\til{pr}.$ Moreover $cw \circ i = 0$ (see Remark \ref{Remark: Calabi-Weinstein}).

This allows us to define the following $8$ conjugation-invariant norms on $\Ham$ (there are analogues for $\til{\Ham},$ which we omit) from the (pseudo-)norms $|\cdot|_\al,|\cdot|^{\osc}_\al,|\cdot|_{str,\al},|\cdot|^{\osc}_{str,\al},$ where the last one is the Finsler restriction from $\G$ to $\cH$ of the Finsler pseudo-metric giving $|\cdot|^{\osc}_\al$ on $\G:$

\begin{gather}\label{Equation: restriction of Hofer}
|\phi|_1 = \inf_{\pi(\til{\phi}) = \phi} |i(\til{\phi})|_{str,\al},\;\; |\phi|_2 =  \inf_{\pi(\til{\phi}) = \phi} |i(\til{\phi})|_{\al},\\
|\phi|_3 = \inf_{\pi(\til{\phi}) = \phi} |i(\til{\phi})|^{\osc}_{str,\al},\;\; |\phi|_4 = \inf_{\pi(\til{\phi}) = \phi} |i(\til{\phi})|^{\osc}_{\al},\\
|\phi|_5 = \inf_{pr(\wh{\phi}) = \phi} |\wh{\phi}|_{str,\al},\;\; |\phi|_6 = \inf_{pr(\wh{\phi}) = \phi} |\wh{\phi}|_{\al},\\
|\phi|_7 = \inf_{pr(\wh{\phi}) = \phi} |\wh{\phi}|^{\osc}_{str,\al},\;\; |\phi|_8 = \inf_{pr(\wh{\phi}) = \phi} |\wh{\phi}|^{\osc}_{\al}.
\end{gather}

By Proposition \ref{Proposition: energy-capacity} and Remark \ref{Remark: Reeb closed} it is seen that these norms are non-degenerate. Moreover, by arguments resembling the proof of Proposition \ref{Proposition: Osc via L infty}, which we omit, one sees that \[|\cdot|_3 = |\cdot|_5=|\cdot|_7 = |\cdot|_{Hofer},\] that \[\frac{1}{2}|\cdot|_{Hofer} \leq |\cdot|_1 \leq |\cdot|_{Hofer},\] that \[\nu_1(\cdot):=|\cdot|_4 = |\cdot|_6=|\cdot|_8 \leq |\cdot|_{Hofer}\] and that \[\nu_2(\cdot):=|\cdot|_2\] satisfies \[\nu_1(\cdot) \leq \nu_2(\cdot) \leq |\cdot|_1 \leq |\cdot|_{Hofer}.\] This leads one to the following questions.

\medskip
\begin{qtn}\label{Question: two new metrics and Hofer}
Is $\nu_1$ equivalent to the Hofer norm? Is $\nu_2$ equivalent to the Hofer norm?
\end{qtn}

\medskip
\begin{qtn}\label{Question: two new metrics and Hofer}
Is $\nu_1$ unbounded? Is $\nu_2$ unbounded?
\end{qtn}

In this direction, using Givental's quasimorphism \cite{GiventalQuasimorphism}, we prove the following statement. 

\medskip
\begin{prop}\label{Proposition: nu2 is unbounded on RP 2n+1}
If $M = \C P^n$ with the symplectic form $2\,\om_{FS},$ where $\om_{FS}$ is the Fubini-Study form normalized so that $\langle [\om_{FS}], [\C P^1]\rangle = 1,$ and $P = \R P^{2n+1}$ with the standard contact form $\al_{st}$ is its prequantization, then $\nu_2$ on $\Ham(\C P^n)$ is unbounded.
\end{prop}
%
%
%
%
\section{Proofs}\label{Section: Proofs}

We open this section with the following general result from \cite[Proof of Theorem 1.3]{UsherObservations}, which we refer to as Usher's trick. We remark that while Usher's trick is originally stated for compactly supported Hamiltonians on general symplectic manifolds, it works equally well for degree $1$ $\R_{>0}$-homogeneous Hamiltonians on the symplectization $SN$ of $N$ and their $\R_{>0}$-equivariant flows. Put $\cH^1_{SN}$ for the space of degree $1$ $\R_{>0}$-homogeneous Hamiltonians in $\sm{[0,1] \times SN,\R}.$

\medskip

\begin{prop}\label{Proposition: Usher's trick}
Given any $H \in \cH^1_{SN},$ there exists $K \in \cH^1_{SN}$ with flow $\{\phi^t_K\}_{t \in [0,1]},$ such that
\begin{enumerate}[label={U\arabic*}]
\item \label{item: Usher property 1} $\{\phi_H^t\}$ and $\{\phi_K^t\}$ have the same endpoints (and are in fact homotopic through $\R_{>0}$-equivariant Hamiltonian flows with fixed endpoints), and
\item for all $t \in [0,1],\; x \in M,$ \label{item: Usher property 2}\[K(t,\phi_K^t(x)) = H(1-t,x).\]    
\end{enumerate}
\end{prop}

We require the following cut-off technique. For an interval $I = (a,b) \subset \R_{>0}$ and $\kappa > 0,$ denote $I^\kappa = (e^{-\kappa} a, e^{\kappa} b) \subset \R_{>0}.$ Note that $I^\kappa \supset I.$

\medskip

\begin{lma}\label{Lemma: cutoff}
Let $\{\overline{\psi}_t\}$ be a path of $\R_{>0}$-equivariant Hamiltonian diffeomorphisms of $SN.$ Let $\overline{H}$ be its degree $1$ homogeneous Hamiltonian. Given an interval $I = (a,b) \subset \R_{>0}$ and $\delta > 0,$ fix a cutoff function $\la_0^{I,\delta}: \R_{>0} \to [0,1]$ that satisfies $\la_0^{I,\delta} \equiv 1$ on $I^{\delta/2},$ and $\la_0^{I,\delta} \equiv 0$ on $\R_{>0} \setminus I^{\delta}.$ Denote by $\la^{I,\delta}:SN \cong N \times \R_{>0} \to [0,1]$ the lift $\la^{I,\de}(x)=\la_0^{I,\delta}(r_\al(x))$ of $\la_0^{I,\delta}$ to $SN$ by the projection $r_\al$ to the second coordinate. Then the new compactly supported Hamiltonian \[\overline{H}^{I,\delta}(t,x) = \overline{H}(t,x) \cdot \la^{I,\delta}((\overline{\psi}_t)^{-1} x)\] with flow $\{\overline{\psi}_t^{I,\delta}\}$ satisfies the following properties. 
\begin{enumerate}[label=C\arabic*]
\item \label{item: cutoff Property 1}  $ \overline{\psi}_t^{I,\delta}|_{I \cdot N} \equiv \overline{\psi}_t|_{I \cdot N_1},$
\item \label{item: cutoff Property 2} $ \int_0^1 |\overline{H}^{I,\delta}_t|_{L^\infty(SN)} dt \leq e^\delta (\sup I) \cdot \int_0^1 |\overline{H}_t \circ \overline{\psi}_t|_{L^\infty(N_1)} dt.$
\end{enumerate}
\end{lma}

Here $\sup I = b.$

\begin{proof}[Proof of Lemma \ref{Lemma: cutoff}]
The proof of Property \ref{item: cutoff Property 1} is simply the fact that for $x \in I \cdot N$ and $t \in [0,1],$ on $\overline{\psi}_t(I \cdot N_1)$ the Hamiltonian vector fields of $\overline{H}^{I,\delta}(t,x)$ and $\overline{H}(t,x)$ coincide, therefore $\overline{\psi}_t(x)$ is a solution of the ODE defining the Hamiltonian flow of $\overline{H}^{I,\delta}$ with initial condition $x \in I\cdot N_1.$ The conclusion follows by uniqueness of solutions to ODE.

The proof of Property \ref{item: cutoff Property 2} is that outside $I^\delta \cdot N_1,$ where $I^\delta = (e^{-\delta}a,e^{\delta}b),$ the cut-off $\la^{I,\delta}$ vanishes. Therefore $\overline{H}_t^{I,\de}$ vanishes outside $\overline{\psi}_t(I^\de \cdot N_1).$ Hence \[\int_0^1 |\overline{H}^{I,\delta}_t|_{L^\infty(SN)} dt \leq \int_0^1 |\overline{H}_t|_{L^\infty(\overline{\psi}_t(I^\de \cdot N_1))} dt = \] \[= \int_0^1 |\overline{H}_t \circ \overline{\psi}_t|_{L^\infty(I^\de \cdot N_1)} dt \leq e^\delta (\sup I) \cdot \int_0^1 |\overline{H}_t \circ \overline{\psi}_t|_{L^\infty(N_1)} dt.\] In the last inequality we used that $\overline{\psi}_t$ is $\R_{>0}$-equivariant and $\overline{H}_t$ is $\R_{>0}$-homogeneous of degree $1,$ and noted that $\sup I^\de = e^\delta \cdot \sup I.$
\end{proof}

\begin{proof}[Proof of Proposition \ref{Proposition: energy-capacity}]
Assume that $\overline{\psi}$ displaces a compact subset $A$ of $SN.$ Let $\psi$ be generated by the contact Hamiltonian $H(t,x)$ with flow $\{\psi_t \}.$ Denote by $\overline{H}(t,x)$ its degree $1$ homogeneous lift to $SN$ with flow $\{\overline{\psi}_t\}$ the canonical lift of $\{\psi_t\}$ to $SN.$ Apply Usher's trick (Propostion \ref{Proposition: Usher's trick}) to $\overline{H}(t,x)$ to obtain a degree $1$ homogeneous Hamiltonian $\overline{K}(t,x),$ satisfying Properties \ref{item: Usher property 1} and \ref{item: Usher property 2}. The time-one map $\overline{\psi} = \psi^1_{\overline{K}}$ of the flow of $\overline{K}$ displaces $A.$ Hence, if we choose $I = (\min_A r_\al, \max_A r_\al),$ then for any $\delta >0$ by Property \ref{item: cutoff Property 1} so does the the time-one map $\psi^1_{\overline{K}^{I,\delta}}$ of the cut-off $\overline{K}^{I,\delta}$ as performed in Lemma \ref{Lemma: cutoff}. Therefore by \cite[Theorem 1.1]{LalondeMcDuffEnergy} of Lalonde-McDuff, we have \[\frac{1}{4} c(A) \leq \int_0^1 |\overline{K}^{I,\delta}_t|_{L^\infty(SN)} \,dt.\] Furthermore, by Properties \ref{item: cutoff Property 2} and \ref{item: Usher property 2} we estimate \[\int_0^1|\overline{K}^{I,\delta}_t|_{L^\infty(SN)} \,dt \leq e^\delta h(A) \cdot \int_0^1 |\overline{K}_t \circ \psi^t_{\overline{K}}|_{L^\infty(N_1)} \, dt = e^\delta h(A) \cdot \int_0^1 |\overline{H}_t|_{L^\infty(N_1)} \, dt,\] recalling that $h(A)=\max_A r_\al.$ As $i_1^* \overline{H}_t = H_t$ for the obvious isomorphism $i_1: N \to N_1,$ we conclude that \[ \frac{1}{4} c(A) \leq e^\delta h(A) \cdot \int_0^1 |H_t|_{L^\infty(N)} \, dt,\] whence by taking infima over contact Hamiltonians $H(t,x)$ generating $\psi$ and over $\delta > 0,$ we obtain \[\frac{1}{4}\widehat{c}_\al(A) \leq |\psi|_\al.\]
 
\end{proof}

\begin{proof}[Proof of Theorem \ref{Theorem: cocycle norm}]
Property \ref{item: metric Property 1} follows from the fact that for $\psi \in \G,$ the condition $\psi \neq 1_N$ is equivalent to $\overline{\psi} \neq 1_{SN},$ and the latter condition implies that $\overline{\psi}$ displaces a ball $B \subset SN$ of positive Gromov radius $c(B) > 0$ and height $h_\al(B)>0,$ and hence by Proposition \ref{Proposition: energy-capacity} we have $|\psi|_\al \geq \frac{1}{4} \widehat{c}(B) > 0.$

For property \ref{item: metric Property 2} assume that $\phi$ is generated by contact Hamiltonian $F_t$ and $\psi$ is generated by $G_t.$  We claim that $\phi\psi$ is generated by contact Hamiltonian $H_t$ with \begin{equation}\label{Equation: triangle inequality} \int_0^1 |H_t|_{L^\infty(N)} \,dt \leq \int_0^1 |F_t|_{L^\infty(N)} \,dt + \int_0^1 |G_t|_{L^\infty(N)} \,dt.\end{equation} Then taking infima first on the left-hand side and then on the right hand side finishes the proof. Note that the functional $F \mapsto \int_0^1 |F_t|_{L^\infty(N)} \,dt$ is invariant under time-reparametrizations $F(t,x) \mapsto \tau'(t) F(\tau(t),x)$ acting as $\{ \phi_t \}_{t\in [0,1]} \mapsto \{ \phi_{\tau(t)} \}_{t\in [0,1]}$ on the flows, for $\tau:[0,1] \to [0,1]$ a smooth function with $\tau(0)=0, \tau(1)=1$ and $\tau' \geq 0.$ Choose such reparametrizations $\tau_1, \tau_2$ with $\mathrm{supp} (\tau'_1) \subset [1/2,1]$ and $\mathrm{supp} (\tau'_2) \subset [0,1/2].$ Then $\phi\psi$ is generated by the flow $\phi_{\tau_1(t)}\psi_{\tau_2(t)}.$ The contact Hamiltonian $H_t$ of this flow satisfies $H_t = \tau'_2(t) G_{\tau_2(t)}$ for $t\in [0,1/2]$ and $H_t = \tau'_1(t) F_{\tau_1(t)}$ for $t \in [1/2,1],$ whence Equation \ref{Equation: triangle inequality} follows immediately.

Property \ref{item: metric Property 3} is an immediate consequence of the fact that if $H(t,x)$ generates $\psi_t$ with endpoint $\psi$ then $-H(1-t,x)$ generates $\psi_{1-t}\psi^{-1}$ with endpoint $\psi^{-1}.$

Property \ref{item: metric Property 4} follows from the fact that if $X_t$ generates $\{\phi_t\}$ with $\phi_1 = \phi,$ then $Y_t = \psi_*(X_t)$ that is $Y_t(x)= (D\psi)(\psi^{-1}x) X_t(\psi^{-1}x)$ generates $\{\psi\phi_t\psi^{-1}\}$ with $\psi\phi_1\psi^{-1} = \psi\phi\psi^{-1},$ and hence if $F_t = \al(X_t)$ is the contact Hamiltonian for $X_t$ then $G_t = \al(\psi_*(X_t))=(\psi^*\al)(X_t) \circ \psi^{-1}$ is the contact Hamiltonian for $Y_t.$ 
\end{proof}


\begin{proof}[Proof of Proposition \ref{Proposition: $C^0$-uniqueness}]
Since $\phi$ is a uniform limit of continuous maps, it is continuous. Assume that $\phi \neq \psi.$ Then $\phi \psi^{-1}$ displaces a small closed ball $B \subset N.$ By uniform convergence, this implies that $\psi_i \psi^{-1}$ displaces $B$ for all $i$ large enough. Therefore by Corollary \ref{Corollary: contact energy-capacity} \[d_\al(\psi_i,\psi) \geq \frac{1}{4} \til{c}_\al(B) > 0,\] in contradiction to assumption \ref{item: C0 Hofer}.
\end{proof}
\begin{proof}[Proof of Proposition \ref{Proposition: example of unboundedness}]
Our approach is based on \cite[Exercise 7.2.E]{P-book} (cf. \cite{LalondeMcduffLinfty}) in the Hamiltonian case. This approach requires certain knowledge about the orbit class evaluation map \[ev: \pi_1(\Cont_0(P_1,\xi)) \to \pi_1(P_1).\]
We first note that $\pi_1(P_1) \cong \mathrm{H}_{\Z},$ the integer Heisenberg group of unipotent upper triangular $3 \times 3$ matrices with integer coefficients. It is easy to see that $\mathrm{H}_{\Z} = \langle x,y,z \,|\; [x,y]= z,\, [z,x] = 1,\, [z,y] = 1 \rangle.$ Therefore the centre $\mathrm{Z}(\pi_1(P_1))$ of $\pi_1(P_1)$ satisfies  $\mathrm{Z}(\pi_1(P_1)) = \langle z \rangle.$ Moreover, $z$ is the orbit class of the Reeb loop. Finally, as is true for all diffeomorphism groups, $\mathrm{image}(ev) \subset \mathrm{Z}(\pi_1(P_1)).$ We conclude that \begin{equation}\label{Equation: evaluation orbit class} \mathrm{image}(ev) \subset \langle z \rangle.\end{equation}

Arguing up to $\varepsilon,$ there is no loss of generality for the purposes of this argument in assuming that the autonomous vector field in the neighbourhood $U=[(-\eps,\eps)] \times S^1$ for small $\eps>0$ of this curve is constant with respect to the stadard trivialization of the tangent bundle on $T^2.$ In this neigbourhood the prequantization can be trivialized as $U \times S^1$ with the contact form $\al_1 = d\theta + xdy,$ where $x:(-\eps,\eps) \to \R, \; y:S^1 \to S^1$ are the standard coordinates, and $\theta: S^1 \to S^1$ the standard coordinate. Hence, over $U$ the symplectization of $P_1$ trivializes as $U \times S^1 \times \R_{>0}$ with the symplectic form $\om = d(r(d\theta + xdy)) = dr \wedge d\theta + d(rx) \wedge dy.$ Putting $X = rx$ for a new coordinate, the symplectic form splits as $\om = dr \wedge d\theta + dX \wedge dy.$ Passing to universal covers we have $\widetilde{P_1|_U} = \widetilde{U} \times \R,$ where $\til{U} = (-\eps,\eps) \times \R$ and $\widetilde{SP_1|_U} = S(\widetilde{P_1|_U}) = \widetilde{U} \times \R \times \R_{>0}$ with the symplectic form $dr \wedge d\theta + d(rx) \wedge dy = dr \wedge d\theta + dX \wedge dy,$ where now $\theta:\R \to \R,$ and $y:\R \to \R$ are the standard coordinates. For a small $\delta>0,$ let $B'(t') = \{(r,\theta) \in \R_{>0} \times \R | \; \delta \leq r \leq 1-\delta, \; 0\leq \theta \leq \frac{t'}{1-2\delta} \} \subset \R \times \R_{>0}$ be a rectangle of capacity (area) $t'$ in the $(r,\theta)$ half-plane. Note that in particular $r \geq \delta$ on $B'(t'),$ and hence the image of $\til{U} \times B'(t)$ in the new coordinates $(X,y,\theta,r)$ contains the strip $[-\eps \cdot \delta , \eps \cdot \delta] \times \R \times B'(t').$ Pick the rectangle $B(t') = \{(X,y)|\; X \in [-\eps \cdot \delta , \eps \cdot \delta], \; y \in [0, \frac{t'}{2\eps \cdot \delta}]\} \subset [-\eps \cdot \delta , \eps \cdot \delta] \times \R$ of capacity $t'.$ Then the Gromov radius of $A(t')=B(t')\times B'(t')$ satisfies \[c(A(t')) \geq t'.\] Consider $A(t')$ as a subset of $\widetilde{SP_1}.$

Fixing $t > 0,$ we claim that \[|\psi^t_H|_{\al_1} \geq c\cdot t,\] for $c = \frac{1}{2} {\mathrm{const}({H|_U}) \cdot \eps \cdot \delta },$ where $ \mathrm{const}({H|_U}) > 0$ is (a lower bound on) the Riemannian length of the Hamiltonian vector field of $H|_U$ with respect to the standard Riemannian metric on $T^2.$

Fix any $t'<4 c\cdot t,$ and let $F(s,x) \in \sm{[0,1] \times P_1,\R}$ be a contact Hamiltonian generating a contact isotopy $\{\psi^s_F\}_{s \in [0,1]}$ with $\psi^1_F = \psi^t_H.$ Consider its deck-invariant lift $\widetilde{F}$ to the universal cover of $P_1$, generating $\widetilde{\psi}^s_F$. We first show that $\widetilde{\psi}^1_F$ displaces $A(t').$ Taking $F_0 = t \pi^* H,$ it is immediate to see that $\widetilde{\psi}^1_{F_0}$ displaces $A(t')$ by considering the projection of $\widetilde{\psi}^1_{F_0} (A(t'))$ to the universal cover of $T^2.$ Any other isotopy $\{\psi^s_F\}$ differs from $\{\psi^s_{F_0}\}$ by a loop $\gamma = \{\gamma_s\}_{s \in [0,1]}$ in $\Cont_0(P_1,\xi),$ based at the identity: that is $\psi^s_F = \gamma_s \psi^s_{F_0}.$ Hence $\widetilde{\psi}^1_{F} = D(ev([\gamma]))\widetilde{\psi}^1_{F_0},$ $D(a)$ denoting the Deck transformation corresponding to an element $a \in \pi_1(P_1).$ Hence by \eqref{Equation: evaluation orbit class}, the projections of $\widetilde{\psi}^1_{F} (A(t'))$ and $\widetilde{\psi}^1_{F_0} (A(t'))$ to the universal cover of $T^2$ coincide, and therefore $\widetilde{\psi}^1_{F}$ also displaces $A(t').$

We proceed by cutting $\widetilde{F}$ off by a function $\la_c: \widetilde{P_1} \to [0,1]$ outside a compact subset containing $\bigcup_{s \in [0,1]} \widetilde{\psi}^s_F (\pi(A(t'))$ in its interior, where $\pi: S(\widetilde{P_1}) \to \widetilde{P_1}$ is the natural projection. We obtain a compactly supported contactomorphism of $\widetilde{P_1}$ with compactly supported Hamiltonian $\til{F}^{\la_c}$ satisfying \[\int_0^1 |\til{F}^{\la_c}_s|_{L^\infty(\widetilde{P_1})} ds \leq \int_0^1 |F_s|_{L^\infty(P_1)} ds.\] Moreover the lift $\overline{\psi}^s_{\til{F}^{\la_c}}$ of the flow of $\til{F}^{\la_c}$ to $S(\widetilde{P_1})$ displaces $A(t').$ Hence, noting that $h(A(t')) < 1,$ by Proposition \ref{Proposition: energy-capacity} we have \[\int_0^1 |\til{F}^{\la_c}_s|_{L^\infty(\widetilde{P_1})} ds \geq \frac{t'}{4}.\] Therefore \[\int_0^1 |F_s|_{L^\infty(P_1)} ds \geq \frac{t'}{4},\] and hence first taking the infimum over all contact Hamiltonians $F_s$ generating $\psi^t_H$ and then taking supremum over all $t'<4 c\cdot t,$ we finish the proof. 
\end{proof}

\begin{proof}[Proof of Proposition \ref{Proposition: Quant is closed}]
Assume $\psi \in \G \setminus \cH,$ and $\phi \in \cH.$ Let $\chi = \psi^{-1} \in \G \setminus \cH.$ Then there exists a point $x \in N$ with $\mu:=\max_N \la_{\chi}^{-1}=\la^{-1}_{\chi}(x) > 1.$ There is a ball $B(x) \subset N$ such that $(\la_\chi^{-1})|_{B(x)} \geq \mu^{3/4}.$
Consider the interval $I=(\mu^{-1/4},\mu^{1/4}).$ We claim that $(\overline{\phi}) (\overline{\psi})^{-1} = (\overline{\phi}) (\overline{\chi})$ displaces $B(x)\times I,$ and therefore denoting $\varepsilon = \frac{1}{4} \widehat{c}_\al (B(x)\times I)$ we have $d_\al(\psi,\phi) \geq \varepsilon > 0$ by Proposition \ref{Proposition: energy-capacity}. Indeed, for $(y,r) \in SN,$ $\overline{\chi}(y,r) = (\chi(y),(\la_\chi)^{-1}(y) r),$ hence while $\max_{B(x) \times I} r = \mu^{1/4},$ \[\min_{(\overline{\phi}\overline{\chi})(B(x) \times I)} r = \min_{\overline{\chi}(B(x) \times I)} r = \mu^{-1/4} \min_{B(x)} (\la_\chi)^{-1} \geq \mu^{3/4-1/4} = \mu^{1/2} > \mu^{1/4}.\] This finishes the proof.
\end{proof}

\begin{proof}[Proof of Propostion \ref{Proposition: Osc via L infty}]
Note that \[d_\al(\psi,\cR) = \inf_{\eta \in \R} |\phi^\eta_R\psi|_\al,\] where we use Properties \ref{item: metric Property 3} and \ref{item: metric Property 4} of Theorem \ref{Theorem: cocycle norm}. Now \[\inf_{\eta \in \R} |\phi^\eta_R\psi|_\al = \inf_{H,b(t)} \int_0^1 |-b(t) + H_t|_{L^\infty(N)} \,dt,\] where $H$ is a contact Hamiltonian with $\phi^1_H = \psi$ and $b:[0,1] \to \R$ is a smooth function. Indeed, for a given $\eta$ and $b$ with $B(t)=-\int_0^t b(s) \,ds$ satisfying $B(1)=\eta$ and any path $\{\chi_t\}$ generating $\phi^\eta_R\psi,$ we can write $\chi_t = \phi_R^{B(t)} \phi^t_H,$ for a contact path $\{\phi^t_H\}$ with contact Hamiltonian $H$ and time-one map $\phi^1_H = \psi.$ For any given time $t,$ \[\inf_{b \in \R}|-b + H_t|_{L^\infty(N)} = \frac{1}{2} \osc_N H_t,\] and in fact the infimum is achieved for $b_*(t)=\frac{1}{2}(\max H_t+\min H_t)$ which is a continuous function of $t,$ and hence \[\inf_H \inf_{b(t)} \int_0^1 |-b(t) + H_t|_{L^\infty(N)} \,dt = \frac{1}{2}\inf_H \int_0^1 \osc_N H_t \, dt = \frac{1}{2} |\psi|^{\osc}_\al.\] This finishes the proof.
\end{proof}

Before we continue with the proof of Theorem \ref{Theorem: translated points filling.} we recall some preliminary notions on the Rabinowitz-Floer functionl \cite{AlbersFrauenfelderLeafwise,AlbersMerryTranslated,AlbersMerryFuchs} and prove several useful auxiliary results.

Consider the exact strong symplectic filling $(M,d\overline{\beta}),\; \overline{\beta}|_N = \al$ of $N.$ Denote by \[(W,\om=d\beta)\] the completion of $(M,d\overline{\beta})$ with a cylindrical end along $N = \del M.$ Note that $SN$ embeds symplectically into $W$ by the flow of the Liouville vector field $L$ defined by $\iota_L d\beta = \beta,$ and that $N$ is a hypersurface of restricted contact type in $(W,d\beta).$ 

Consider a compactly supported Hamiltonian $G \in C^\infty_c([0,1] \times W, \R),$ and a defining function $F=F^{\kappa',\kappa}$ with $\kappa' > \kappa >0$ for $N$ given by $F|_{W \setminus SN} \equiv c_-,$ and $F|_{SN} \equiv c_-$ for $r < e^{-\kappa'},$ $F|_{SN} \equiv c_+$ for $r>e^{\kappa'},$ where $c_-<0<c_+$ are appropriately chosen constants, while $F|_{SN} \equiv r-1$ on $N \times (e^{-\kappa},e^{\kappa}).$

Choose reparametrization functions $\tau_1,\tau_2:[0,1] \to [0,1]$ as in the proof of Theorem \ref{Theorem: cocycle norm}, that is for $j\in {1,2},$ $\tau_j(t) \equiv 0$ near $t = 0,$ $\tau_j(t) \equiv 1$ near $t =1,$ $\tau'_j(t) \geq 0$ for all $t \in [0,1]$ and $\mathrm{supp} (\tau'_1) \subset [1/2,1],$ $\mathrm{supp} (\tau'_2) \subset [0,1/2].$ 


Let $\cL W$ denote the loop space of $W.$ Following \cite{AlbersFrauenfelderLeafwise,AlbersMerryTranslated,AlbersMerryFuchs} define the (perturbed) Rabinowitz-Floer functional\footnote{The sign differences between the functional here and in \cite{AlbersFrauenfelderLeafwise} are explained by differing conventions for the Hamiltonian vector field $X_G$ of a Hamiltonian $G$ and by our using $(-\eta)$ as the Lagrange multiplier. The convention we use is $\iota_{X_G} d\beta = - dG.$} \[\cA^F_G: \cL W \times \R \to \R\] by \[\cA^F_G(x,\eta) = -\int_0^1 x^* \beta + \int_0^1 \tau'_1(t) \cdot G(\tau_1(t),x(t))\,dt - \eta \int_0^1 \tau'_2(t) F(x(t)) \,dt.\]

The compactly supported Hamiltonians we use are of the form $\overline{H}^{I,\delta}$ extended by $0$ to $W \setminus SN,$ for $H$ a contact Hamiltonian on $N.$ Given such a Hamiltonian we choose $\kappa$ so that $\mathrm{supp}(\overline{H}^{I,\delta}) \subset [0,1] \times N \times (e^{-\kappa},e^{\kappa}).$ Then we have the following.

\medskip
\begin{lma}\label{Lemma: cutoff Rabinowitz} 
Let $H \in \sm{[0,1]\times N, \R}$ be a contact Hamiltonian. Given any interval $I$ containing $1$ and any $\delta>0,$ the critical points of the perturbed Rabinowitz functional $\cA^F_{\overline{H}^{I,\delta}}$ on $\cL{W} \times \R$ (and consequently on  $\cL{SN} \times \R$) coincide with the critical points of $\cA^F_{\overline{H}}$ on $\cL{SN} \times \R.$ Moreover a critical point $(x_0,\eta_0)$ is Morse for $\cA^F_{\overline{H}^{I,\delta}}$ if and only if it is Morse for $\cA^F_{\overline{H}}.$
\end{lma}

\begin{proof}[Proof of Lemma \ref{Lemma: cutoff Rabinowitz}]
The first statement is a direct consequence of Lemma \ref{Lemma: cutoff} property \ref{item: cutoff Property 1}, because each leafwise-intersection point lies on $N_1.$ Moreover, the Hessians of the two functionals at such a common critical point agree, whence the second statement is immediate.
\end{proof}

Moreover, a computation that is by now standard (compare \cite[Lemma 2.2]{AlbersMerryTranslated}) shows that $(x_0,\eta_0)$ is a critical point of $\cA^F_{\overline{H}}$ if and only if $x_0 \in N_1 \cong N$ and $(x_0,\eta_0)$ is an algebraic translated point of the contact flow on $N$ generated by $H,$ and $\cA^F_{\overline{H}}(x_0,\eta_0) = \eta_0.$ 

\begin{proof}[Proof of Lemma \ref{Lemma: non-degenerate contactomorphism}]
We begin by recalling useful formulas from \cite{AlbersFrauenfelderLeafwise}. Let $(x_0,\eta_0)$ be a critical point of $\cA^F_{\overline{H}}.$ For a Hamiltonian $P \in C^\infty_c([0,1] \times SN, \R)$ put \[\cL_P SN = \{w \in W^{1,2}([0,1],SN)|\, w(0) = \phi^1_P(w(1))\}\] for the twisted loop space, and consider the diffeomorphism \[\Phi_P:\cL_P SN \to \cL SN,\] given by \[x(t) \mapsto \phi^t_P(x(t)).\] Use $\Phi_{-\eta_0 F + \overline{H}}$ to pull back $\cA^F_{\overline{H}}$ to \[\til{\cA}^F_{\eta_0,\overline{H}} = (\Phi_{-\eta_0 F + \overline{H}} \times 1_\R)^* \cA^F_{\overline{H}}: \cL_{-\eta_0 F +\overline{H}} SN \times \R \to \R.\] Set $w_0:=\Phi^{-1}_{-\eta_0 F + \overline{H}}(x_0) = \mathrm{const}.$ By \cite[Equation A.11]{AlbersFrauenfelderLeafwise} we see that the Hessian of $\til{\cA}^F_{\eta_0,\overline{H}}$ at $(w_0,\eta_0)$ is given by \[\cH_{\til{\cA}^F_{\eta_0,\overline{H}}}(w_0,\eta_0)[(\widehat{w}_1,\widehat{\eta}_1),(\widehat{w}_2,\widehat{\eta}_2)] = \int_0^1 \om(\del_t \widehat{w}_1,\widehat{w}_2) - \wh{\eta}_1 \int_0^1 \tau'_2(t) dF(w_0)[\wh{w}_2] - \wh{\eta}_2 \int_0^1 \tau'_2(t) dF(w_0)[\wh{w}_1].\]

Choosing an inner product on $T_{w_0}(\cL_{-\eta_0 F +\overline{H}} SN) \times T_{\eta_0}\R \cong T_{w_0}(\cL_{-\eta_0 F +\overline{H}} SN) \times \R$ one turns $\cH_{\til{\cA}^F_{\eta_0,\overline{H}}}(w_0,\eta_0)$ to an operator $L_{(w_0,\eta_0)}: T_{w_0}(\cL_{-\eta_0 F +\overline{H}} SN) \times \R \to (\mathcal{E}_{-\eta_0 F +\overline{H}})_{w_0}^* \times \R,$ where \[(\mathcal{E}_{-\eta_0 F +\overline{H}})_{w_0} = L^2([0,1], {w_0}^* T(SN)),\] that is a Fredholm operator of index $0,$ since its extension to $L^2$ is self-adjoint. Therefore, to show that $\til{\cA}^F_{\eta_0,\overline{H}}$ is Morse at $(w_0,\eta_0),$ or equivalently $\cA^F_{\overline{H}}$ is Morse at $(x_0,\eta_0),$ it is necessary and sufficent that the kernel of $\cH_{\til{\cA}^F_{\eta_0,\overline{H}}}(w_0,\eta_0)$ be trivial.

Let $(\wh{w}_1,\wh{\eta}_1)$ be such that $\cH_{\til{\cA}^F_{\eta_0,\overline{H}}}(w_0,\eta_0)[(\widehat{w}_1,\widehat{\eta}_1),(\widehat{w}_2,\widehat{\eta}_2)] = 0$ for all $(\widehat{w}_2,\widehat{\eta}_2).$ Note that \[\wh{w}_1, \wh{w}_2 \in W^{1,2}([0,1],T_{w_0} SN)\] with the condition that \begin{eqnarray}\label{Equation: variations of twisted loops}\wh{w}_1(0)= D\overline{\phi}_{-\eta_0 F + \overline{H}}(\wh{w}_1(1)), \\ \wh{w}_2(0)= D\overline{\phi}_{-\eta_0 F + \overline{H}}(\wh{w}_2(1)) \nonumber.\end{eqnarray} Noting that $w_0 \in N_1,$ we consider the splitting \[T_{w_0} (SN) \cong \R\langle R \rangle \oplus \R\langle \del_r \rangle \oplus \xi_{w_0},\] and write \[\wh{w}_1(t) = a_1(t)R + b_1(t) \del_r + \wh{w}^\xi_1(t),\] \[\wh{w}_2(t) = a_2(t)R + b_2(t) \del_r + \wh{w}^\xi_2(t)\] with respect to this splitting. Plugging in different $(\wh{w}_2,\wh{\eta}_2)$ - of the form $(0,\wh{\eta}_2),$ $(b_2 \del_r,0),$ $(a_2 R, 0),$ and $(\wh{w}^\xi_2,0),$ we obtain the identities \begin{enumerate}
\item \label{item: Morse RF 1} $\int_0^1 \tau'_2(t)b_1(t) \,dt = 0,$
\item \label{item: Morse RF 2} $a'_1 + \wh{\eta}_1 \tau'_1 \equiv 0,$
\item \label{item: Morse RF 3} $b'_1 \equiv 0,$
\item \label{item: Morse RF 4} $\del_t \wh{w}^\xi_1 \equiv 0.$
\end{enumerate}

From identities \ref{item: Morse RF 1} and \ref{item: Morse RF 3}, we conclude that $b_1 \equiv 0,$ as $\int_0^1 \tau'_2(t) \, dt = 1.$ Identity \ref{item: Morse RF 4} means that $\wh{w}^\xi_1$ is a constant vector. Identity \ref{item: Morse RF 2} implies that \begin{equation}\label{Equation: Morse RF difference in a} a_1(1) - a_1(0) = - \wh{\eta}_1.\end{equation} Now Equation \ref{Equation: variations of twisted loops} implies that \[a_1(0) R + \wh{w}^\xi_1(0) = D(\overline{\phi}_{-\eta_0 F + \overline{H}})(a_1(1)R + \wh{w}^\xi_1(1))\] which, noting that $\phi^{-\eta_0}_R w_0 = x_0,$ is equivalent by Equation \ref{Equation: differential of the lift} to the conditions \[d\la_{\psi \phi^{-\eta_0}_R}(w_0)(a_1(1)R + \wh{w}^\xi_1(1)) = d\la_{\psi}(x_0) \circ D(\phi^{-\eta_0}_R)(w_0)(a_1(1)R + \wh{w}^\xi_1(1)) = 0,\]  and \[a_1(0) R + \wh{w}^\xi_1(0) = D({\phi}_{-\eta_0 F + \overline{H}})(a_1(1)R + \wh{w}^\xi_1(1)).\] Evaluating $\al_{w_0}$ on both sides of the equality, and noting that $\la_\psi(x_0) = 1,$ we obtain \begin{equation}\label{Equation a_1(0) equals a_1(1)}
a_1(0) = \al_{w_0}(D({\phi}_{-\eta_0 F + \overline{H}})(a_1(1)R + \wh{w}^\xi_1(1))) = \la_\psi(x_0) \cdot \al_{w_0}(a_1(1)R + \wh{w}^\xi_1(1))) = a_1(1).
\end{equation}

Equations (\ref{Equation: Morse RF difference in a}) and (\ref{Equation a_1(0) equals a_1(1)}) imply that $\wh{\eta}_1 = 0,$ and therefore $a'_1 \equiv 0.$ This means that $a_1$ and therefore $\wh{w}_1 = a_1 R + \wh{w}_1^\xi$ is a constant vector, which we denote by $\theta_{w_0}.$ We conclude that $\theta_{w_0} \mapsto D(\phi^{-\eta_0}_R)(w_0)(\theta_{w_0})$ induces an injection $\iota(w_0,\eta_0)$ of $\ker \cH_{\til{\cA}^F_{\eta_0,\overline{H}}}(w_0,\eta_0)$ into $\ker (D(\phi^{-\eta_0}_R\psi)(x_0) - 1_{TN_{x_0}}) \cap \ker (d\la_\psi(x_0)).$ The Lemma follows immediately.
\end{proof}

\begin{rmk}
Retracing the steps in the above argument we easily see that the injection $\iota(w_0,\eta_0)$ is in fact surjective, and hence an isomorphism of vector spaces \[\ker \cH_{\til{\cA}^F_{\eta_0,\overline{H}}}(w_0,\eta_0) \cong \ker (D(\phi^{-\eta_0}_R\psi)(x_0) - 1_{TN_{x_0}}) \cap \ker (d\la_\psi(x_0)).\]
\end{rmk}

\begin{proof}[Proof of Theorem \ref{Theorem: translated points filling.}]
Given a contactomorphism $\psi$ we note that for each $\eta \in \R$, $x$ is a translated point of $\psi$ if and only if $x$ is a translated point of $\phi^\eta_R \psi.$ Moreover by the definition of non-degeneracy (Definition \ref{Definition: non-degenerate contactomorphism}) $(x,\eta_0)$ is a non-degenerate algebraic translated point of $\psi$ if and only if $(x,\eta_0+\eta)$ is a non-degenerate algebraic translated point of $\phi^\eta_R \psi.$ If $\psi$ is of small oscillation energy there exists $\eta_* \in \R$ such that $|\phi^{\eta_*}_R \psi|_\al < \rho(\al).$ Therefore there exists a Hamiltonian $H'$ generating $\{\phi^t_{H'}\}$ with $\phi^1_{H'} = \phi^{\eta_*}_R \psi$ with $\int_0^1 |H'_t|_{L^\infty(N)} \,dt < \rho(\al).$ Then by Usher's trick there exists a cutoff $\overline{H}^{I,\delta}$ with $I=(e^{-\eps},e^\eps)$ for $\eps > 0$ of the degree $1$ homogeneous lift $\overline{H}$ to $SN$ of another contact Hamiltonian $H$ with $\phi^1_{H} = \phi^{\eta_*}_R \psi$ such that \[\int_0^1 |\overline{H}^{I,\delta}|_{L^\infty(SN)} \,dt < \rho(\al).\] Since $M$ is a strong exact symplectic filling of $(N,\al)$ we can consider $\overline{H}^{I,\delta}$ as a function on $W$ and apply \cite[Theorem A]{AlbersFrauenfelderLeafwise} to see that the Hamiltonian flow of $\overline{H}^{I,\delta}$ possesses a leafwise-intersection point relative to the hypersurface $N_1 \subset W.$ By Lemma \ref{Lemma: cutoff Rabinowitz} this gives us a translated point of $\phi^{\eta_*}_R \psi$ and therefore a translated point of $\psi.$ If $\psi$ is non-degenerate, then so is $\phi^{\eta_*}_R \psi,$ and therefore $\overline{H}^{I,\delta}$ is non-degenerate relative to $N_1 \subset W.$ Therefore \cite[Proposition 2.20]{AlbersFrauenfelderLeafwise} and Lemma \ref{Lemma: cutoff Rabinowitz} finish the proof (see also \cite[Theorem B]{AlbersFrauenfelderLeafwise}).
\end{proof}

\begin{proof}[Proof of Proposition \ref{Proposition: nu2 is unbounded on RP 2n+1}]
Let $c_1,c_2,c>0$ denote generic positive constants, and $c_0 \neq 0$ denote a generic non-zero constant. Consider the Givental quasimorphism \cite{GiventalQuasimorphism} \[\mu_G: \til{\G} \to \R,\] where $\G = \Cont_0(\R P^{2n+1},\xi_{st}),$ for $\xi_{st} = \ker \al_{st}.$ The main result of \cite{GabiGiventalCalabi} by Ben Simon applied to this setting states that for any open ball $B$ in $\C P^n$ that is displaceable by a Hamiltonian diffeomorphism, the restriction of the Givental quasimorphism to $\til{\Ham}_c(B)$ by the chain of natural maps $\til{\Ham}_c(B) \to \til{\Ham}(\C P^n) \xrightarrow{i} \til{\cH} \to \til{\G}$ is equal to $c_0 \cdot Cal_B,$ where $Cal_B(\til{\phi})$ is defined as \[ \int_0^1 dt \int_B H_t \,\om_B^n \] for $\om_B = 2\,\om_{FS}|_{B}$ and $H \in C^\infty_c([0,1]\times B,\R)$ any Hamiltonian normalized to vanish near $\del B$ generating a path $\{\phi_t\}_{t=0}^1$ whose class in $\til{\Ham}_c(B)$ is $\til{\phi}.$ In particular $|\mu_G|$ restricted to the image of the composition $\til{\Ham}(\C P^n) \xrightarrow{i} \til{\cH} \to \til{\G}$ is unbounded. We recall that the image of the map $\til{\Ham}(\C P^n) \xrightarrow{i} \til{\cH}$ lies in the kernel of the Calabi-Weinstein homomorphism $\til{\cH} \to \R$ (see Remark \ref{Remark: Calabi-Weinstein}). Moreover, by \cite[Section 2.2]{RemarksInvariantsLoops} for all $\gamma \in \pi_1(\Ham(\C P^n)),$ the element $i(\gamma^{n+1}) = i(\gamma)^{n+1} \in \til{\cH}$ in fact lies in $\pi_1(\cH).$ Moreover, as $\mu_G$ is homogeneous, by \cite[Theorem 1]{RemarksInvariantsLoops} we have \[\mu_G(i(\gamma)) = \frac{1}{n+1}\mu_G(i(\gamma)^{n+1}) = 2\, cw(i(\gamma)^{n+1}) = 0.\] Hence, for $\til{\phi} \in \til{\Ham}(\C P^n)$ and for any $\gamma \in \pi_1(\Ham(\C P^n)),$ as $\mu_G$ is homogeneous and $i(\gamma) \in Z(\til{\cH})$ is a central element, we have $\mu_G(i(\til{\phi}\gamma)) = \mu_G(i(\til{\phi})) + \mu_G(i(\gamma)) = \mu_G(i(\til{\phi})),$ and hence \[\overline{\mu}_G(\phi):=\mu_G(i(\til{\phi}))\] depends only on $\phi = \pi(\til{\phi}) \in \Ham(\C P^n).$ Moreover $|\overline{\mu}_G|$ is unbounded on $\Ham(\C P^n).$

Now Proposition \ref{Proposition: upper bound} Item \ref{item: upper bound 4} implies that for $\til{\phi} \in \til{\Ham}(\C P^n)$ with $\pi(\til{\phi}) = \phi,$ \[|\overline{\mu}_G(\phi)| = |\mu_G(i(\til{\phi}))|\leq c_1\, |i(\til{\phi})|_\al + c_2,\] and hence \[|\overline{\mu}_G(\phi)| \leq c_1\,\nu_2(\phi) + c_2,\] as $\nu_2(\phi) = \inf_{\pi(\til{\phi}) = \phi} |i(\til{\phi})|_\al.$ This implies that $\nu_2$ is unbounded, as required.


\end{proof}

\section*{Acknowledgements}
I thank Sheila Margherita Sandon and Leonid Polterovich for introducing me to the main subjects of this paper over the course of numerous conversations. I thank Strom Borman, Viktor Ginzburg, Michael Khanevsky, Boris Khesin, Fran{\c{c}}ois Lalonde, Dmitry Tonkonog, Michael Usher and Frol Zapolsky for useful conversations. I thank Strom Borman, Leonid Polterovich, Sheila Margherita Sandon and Frol Zapolsky for their comments on an earlier version of this manuscript. The work on this paper has started right after the workshop "Rigidity and Flexibility in Symplectic Topology and Dynamics" at the Lorentz Center (Leiden, 2014). I thank its organizers and participants for a very enjoyable and stimulating event. This work was carried out at CRM, University of Montreal, and I thank this institution for its warm hospitality.


\bibliographystyle{amsplain}
\bibliography{HoferEnergyContactomorphism}




\end{document}